\documentclass{amsart}
\usepackage{amsfonts}
\usepackage{amsmath}
\usepackage{amssymb}
\usepackage{amscd}
\usepackage{mathrsfs}  % This package allows the use of script letter. Type \mathscr to invoke math script.
\usepackage{amsbsy}
\usepackage{amsthm}
\usepackage{graphicx}
\usepackage{fancyhdr}
\usepackage{epsfig}
\usepackage{color}
\usepackage{xy}
\usepackage{CJKutf8}
\usepackage{inputenc}
\usepackage[encapsulated]{CJK}
\usepackage[CJK, overlap]{ruby}

\usepackage[OT2,OT1]{fontenc}  % This package allows the use of cyrillic fonts. They are invoked by typing \textcyr{...}
\newcommand\cyr{%
\renewcommand\rmdefault{wncyr}%
\renewcommand\sfdefault{wncyss}%
\renewcommand\encodingdefault{OT2}%
\normalfont \selectfont} \DeclareTextFontCommand{\textcyr}{\cyr}

\newcommand{\be}{\begin{equation}}
\newcommand{\ee}{\end{equation}}

\newcommand{\inn}[2]{{\langle #1,#2 \rangle}}

\newcommand{\E}{\mathbb{E}}
\newcommand{\bH}{\mathbb{H}}

\newcommand{\N}{\mathbb{N}}

\newcommand{\R}{\mathbb{R}}

\newcommand{\X}{\mathbb{X}}

\newcommand{\Z}{\mathbb{Z}}

\newcommand{\A}{\mathbb{A}}

\newcommand{\cS}{\mathcal{S}}
\newcommand{\bS}{\boldsymbol{S}}

\newcommand{\cW}{\mathcal{W}}

\newcommand{\sinc}{\mathop{\mathrm{sinc}}}
\newcommand{\sF}{\mathscr{F}}

\newcommand{\st}{\,\vert\,}

\newcommand{\wW}{\widetilde{\mathcal{W}}}

\newcommand{\bmu}{\boldsymbol{\mu}} % added
\newcommand{\blambda}{\boldsymbol{\lambda}} % added
\newcommand{\bs}{{\boldsymbol{s}}} %added
\newcommand{\cF}{\mathcal{F}}

\newcommand{\fB}{\mathfrak{B}}

\newcommand{\ff}{\mathfrak{f}}

\newcommand{\fb}{\mathfrak{b}}

\newcommand{\nl}{\vskip 10pt\noindent}
\newcommand{\ml}{\noindent\vskip 5pt}

\newcommand{\Aff}{{\textrm{Aff}}}
\newcommand{\card}{\mathrm{card }}

\renewcommand{\rmdefault}{cmr} % Arial
\renewcommand{\sfdefault}{cmr} % Arial
\newtheorem{theorem}{Theorem}
\theoremstyle{plain}

\newtheorem{corollary}{Corollary}

\newtheorem{definition}{Definition}
\newtheorem{example}{Example}

\newtheorem{lemma}{Lemma}

\newtheorem{proposition}{Proposition}
\newtheorem{remark}{Remark}

\numberwithin{equation}{section}
\begin{document}
\title[Fractal Hypersurfaces, Wavelet Sets and Affine Weyl Groups]{Fractal Hypersurfaces, Wavelet Sets\\ and Affine Weyl Groups}
\author{Peter R. Massopust}
\address{Centre of Mathematics, Research Unit M6, Technische Universit\"at M\"unchen, Boltzmannstrasse 3, 85747
Garching b. M\"unchen, Germany, and Helmholtz Zentrum M\"unchen,
Ingolst\"adter Landstrasse 1, 85764 Neuherberg, Germany}
\email{peter.massopust@helmholtz-muenchen.de, massopust@ma.tum.de}

\begin{abstract}
In these lecture notes we present connections between the theory of iterated function systems, in particular those attractors that are graphs of multivariate real-valued fractal functions, foldable figures and affine Weyl groups, and wavelet sets. 
\end{abstract}

\maketitle
\section{Introduction}
These lecture notes deal with the connections between iterated function systems, in particular, multivariate real-valued fractal functions, root systems and affine Weyl groups, and wavelet sets. After a first superficial glance, these areas seem to be too different to contain commonalities. However, the common multiscale structure that appears both in the construction of fractal sets and wavelets points the way to a deeper connection. For instance, it was first shown in \cite{GHM} and \cite{dghm} that a class of wavelets may be constructed by piecing fractal functions together, and then later it was proved in \cite{doug} that every compactly supported refinable function, i.e., every compactly supported scaling function or wavelet, is a piecewise fractal function. The investigation into the multiscale structure of fractals and wavelets was carried out in \cite{m04} and led to the insight that the classical wavelet set concept, which is built on dilation and translation groups, may be adapted to dilation and reflection groups.

A first construction of this new type of wavelet set appeared in \cite{lm} and then some additional insights were reported in \cite{lmo}. The latter two investigations were based on earlier results in \cite{ghm1,ghm2} that had connected the known concepts of multiresolution analysis and affine fractal surface construction to foldable figures, which are in one-to-one correspondence with the fundamental domains of affine Weyl groups \cite{hw}.

Here, we will revisit some of the theoretical background and present some of the main ideas that underlie the construction of dilation-reflection wavelet sets. In order to keep the presentation as self-contained as possible, we first present an updated view of iterated function systems and give a construction of (affine) fractal hypersurfaces that is based on later requirements. These themes make up the contents of Sections \ref{sec2}, \ref{sec3}, and \ref{sec4}. Then we describe root systems, affine reflections, the associated affine Weyl groups, and the concept of foldable figure in Section \ref{sec5}. There, we also show that based on the results in this section, one can construct orthonormal bases of $L^2(\R^n)$ consisting of affinely generated fractal function bases. In Section \ref{sec6}, we introduce the classical wavelet sets. This is done first in the one-dimensional setting and then generalized to $\R^n$. Finally, we define wavelet sets based on dilation groups and affine Weyl groups coming from a foldable figure, and prove their existence for all expansive dilation matrices and all affine Weyl groups.
\section{Iterated Function Systems}\label{sec2}

In this section, we introduce the concept of iterated function system (IFS) and highlight some of its fundamental properties. For more details and proofs, we refer the reader to \cite{barn93,barnd,bv8,hutch} and the references stated therein. 

Throughout this paper, we use the following notation. The set of positive integers is denoted by $\mathbb{N} := \{1, 2, 3, \ldots\}$ and the ring of integers by $\Z$. We denote the closure of a set $S$ by $\overline{S}$ and its interior by $\overset{\circ}{S}$. $(\mathbb{X},d_\X)$ always denotes a complete metric space with metric $d_{\mathbb{X}}$.

\begin{definition}
Let $N\in\mathbb{N}$. If $f_{n}:\mathbb{X}\rightarrow\mathbb{X}$,
$n=1,2,\dots,N,$ are continuous mappings, then $\mathcal{F} :=\left(
\mathbb{X};f_{1},f_{2},...,f_{N}\right)  $ is called an \textbf{iterated
function system} (IFS).
\end{definition}

By slight abuse of terminology we use the same symbol $\mathcal{F}$ for the
IFS, the set of functions in the IFS, and for the following set-valued mapping. We
define $\mathcal{F}:2^{\mathbb{X}}\rightarrow 2^{\mathbb{X}}$ by
\[
\mathcal{F}(B) := \bigcup_{f\in\mathcal{F}}f(B)
\]
for all $B\in2^{\mathbb{X}},$ the set of subsets of $\mathbb{X}$. 

Let $\mathbb{H=H(X)}$ be the set of all nonempty compact subsets of $\mathbb{X}$. As $(\X,d_\X)$ is complete, $(\bH,d_\bH)$ becomes a complete metric space when endowed with the Hausdorff metric $d_{\bH}$ (cf. \cite{Engel})
\[
d_\bH (A,B) := \max\{\max_{a\in A}\min_{b\in B} d_\X (a,b),\max_{b\in B}\min_{a\in A} d_\X (a,b)\}.
\]
Since $\mathcal{F}\left(  \mathbb{H}\right)  \subset\mathbb{H}$, we can also treat $\mathcal{F}$ as a mapping $\mathcal{F}:\mathbb{H} \rightarrow \mathbb{H}$. When
$U\subset\mathbb{X}$ is nonempty, we may write $\mathbb{H}(U)=\mathbb{H(X)}%
\cap2^{U}$. We denote by $\left\vert \mathcal{F}\right\vert $ the number of
distinct mappings in $\mathcal{F}$.

A metric space $\mathbb{X}$ is termed \textbf{locally compact} if for every compact $C\subset\mathbb{X}$ and every positive $r\in\R$ the set $\overline{C+r}$ is again compact. The notation $\overline{C+r}$ means the closure
of the union of balls of radius $r$, one centered on each point of $C$.

The following information is foundational. A proof of it may be found in \cite{bm}.

\begin{theorem}
\label{ctythm}
\begin{itemize}
\item[(i)] If $(\mathbb{X},d_{\mathbb{X}})$ is compact then $(\mathbb{H}%
,d_{\mathbb{H}})$ is compact.

\item[(ii)] If $(\mathbb{X},d_{\mathbb{X}})$ is locally compact then $(\mathbb{H}%
,d_{\mathbb{H}})$ is locally compact.

\item[(iii)] If $\mathbb{X}$ is locally compact, or if each $f\in\mathcal{F}$ is
uniformly continuous, then $\mathcal{F}:\mathbb{H\rightarrow H}$ is continuous.

\item[(iv)] If $f:\mathbb{X\rightarrow}\mathbb{X}$ is a contraction mapping for each
$f\in\mathcal{F}$, then $\mathcal{F}:\mathbb{H\rightarrow H}$ is a contraction mapping.
\end{itemize}
\end{theorem}
\noindent
For $B\subset\mathbb{X}$, let $\mathcal{F}^{k}(B)$ denote the $k$-fold
composition of $\mathcal{F}$, i.e., the union of $f_{i_{1}}\circ f_{i_{2}%
}\circ\cdots\circ f_{i_{k}}(B)$ over all finite words $i_{1}i_{2}\cdots i_{k}$
of length $k.$ Define $\mathcal{F}^{0}(B) := B.$

\begin{definition}
\label{attractdef}A nonempty compact set $A\subset\mathbb{X}$ is said to be an
\textbf{attractor} of the IFS $\mathcal{F}$ if
\begin{itemize}
\item[(i)] $\mathcal{F}(A)=A$ and

\item[(ii)] there exists an open set $U\subset\mathbb{X}$ such that $A\subset U$ and
$\lim_{k\rightarrow\infty}\mathcal{F}^{k}(B)=A,$ for all $B\in\mathbb{H(}U)$,
where the limit is with respect to the Hausdorff metric.
\end{itemize}
The largest open set $U$ such that $\mathrm{(ii)}$ is true is called the \textbf{basin of
attraction} (for the attractor $A$ of the IFS $\mathcal{F}$).
\end{definition}

Note that if $U_1$ and $U_2$ satisfy condition $\mathrm{(ii)}$ in Definition 2 for the same attractor $A$ then so does 
$U_1 \cup U_2$. We also remark that the invariance condition $\mathrm{(i)}$ is not needed; it follows from $\mathrm{(ii)}$ for $B := A$.

We will use the following observation \cite[Proposition 3 (vii)]{lesniak},
\cite[p.68, Proposition 2.4.7]{edgar}.

\begin{lemma}
\label{intersectlemma}Let $\left\{  B_{k}\right\}  _{k=1}^{\infty}$ be a
sequence of nonempty compact sets such that $B_{k+1}\subset B_{k}$, for all
$k\in\N$. Then $\cap_{k\geq1}B_{k}=\lim_{k\rightarrow\infty}B_{k}$ where
convergence is with respect to the Haudorff metric $d_\bH$.
\end{lemma}

The next result shows how one may obtain the attractor $A$ of an IFS. For the proof, we refer the reader to \cite{bm}. Note that we do not assume that the functions in the IFS $\cF$ are contractive.

\begin{theorem}
\label{attractorthm}Let $\mathcal{F}$ be an IFS with attractor $A$ and basin
of attraction $U.$ If the mapping $\mathcal{F}:\mathbb{H(}U)\mathbb{\rightarrow H(}U)$ is
continuous then%
\[
A=\bigcap\limits_{K\geq1}\overline{\bigcup_{k\geq K}\mathcal{F}^{k}(B)},%
\quad\text{for all }B\subset U\text{ such that }\overline{B}\in
\mathbb{H(}U)\text{.}%
\]

\end{theorem}

The quantity on the right-hand side here is sometimes called the
\textbf{topological upper limit }of the sequence $\left\{ \cF^{k}(B)\st k\in \N\right\}$. (See, for instance, \cite{Engel}.) 

A subclass of IFSs is obtained by imposing additional conditions on the functions that comprise the IFS. The definition below introduces this subclass.

\begin{definition}
An IFS $\cF = (\X; f_1, f_2, \ldots, f_N)$ is called \textbf{contractive} if each $f\in \cF$
is a contraction (with respect to the metric $d$), i.e., there is a constant $c \in [0, 1)$
such that 
$$
d(f(x_1), f(x_2)) \leq c\,d(x_1, x_2),
$$
for all $x_1, x_2 \in \X$. 
\end{definition}
By item $\mathrm{(iv)}$ in Theorem 1, the mapping 
$\cF : \mathbb{H} \to \mathbb{H}$ is then also contractive on the complete metric space $(\mathbb{H}, d_{\mathbb{H}})$ and thus -- by the Contraction Mapping Theorem -- possesses a unique attractor $A$. This attractor satisfies the \textbf{self-referential equation}
\be\label{self}
A = \cF(A) = \bigcup_{f\in\mathcal{F}}f(A).
\ee
In the case of a contractive IFS, the basin of attraction for $A$ is $\X$ and the attractor can be computed via the following procedure: Suppose $K_0$ is any set in $\in \bH(\X)$. Consider the sequence of iterates
\[
K_m := \cF(K_{m-1}) = \cF^m (K_0), \quad m\in \N.
\]  
Then $K_m$ converges in the Hausdorff metric to the attractor $A$ as $m\to\infty$, i.e., $d_\bH(K_m, A) \to 0$ as $m\to\infty$.

For the remainder of this paper, we deal exclusively with contractive IFSs as defined above. We will see that the self-referential equation \eqref{self} plays a fundamental role in the construction of fractals sets, i.e., the attractors of IFSs, and in the determination of their geometric and analytic properties.
\section{Fractal Hypersurfaces in $\R^{n+1}$}\label{sec3}
Here, we construct a class of special attractors of IFSs, namely attractors that are the graphs of bounded functions $\ff:\Omega\subset\R^n \to \R$, where $\Omega\in \bH(\R^n)$, $n\in \N$. To this end, we assume that $1 < N\in \N$ and set $\N_N := \{1, \ldots, N\}$. Suppose that $\{u_i : \Omega\to \Omega \st i \in \N_N\}$ is a finite family of bounded bijective mappings with the property that
\begin{enumerate}
\item[(P)] $\{u_i(\Omega)\st i \in\N_N\}$ forms a set-theoretic partition of $\Omega$, i.e., $\Omega = \bigcup_{i=1}^N u_i(\Omega)$ and $u_i(\Omega)\cap u_j(\Omega) = \emptyset$, for all $i\neq j\in \N_N$.
\end{enumerate}
\noindent

We introduce the set $B(\Omega) := B(\Omega, \R) := \{f : \Omega\to \R \st \text{$f$ is bounded}\}$ and endowed it with the metric 
\[
d(f,g): = \displaystyle{\sup_{x\in \Omega}} \,|f(x) - g(x)|.
\] 
It is straight-forward to show that $(B(\Omega), d)$ is a complete metric space. Indeed it is even a complete \textbf{metric linear space}. (A {metric linear space} \cite{Rol} is a vector space endowed with a metric under which the operations of vector addition and scalar multiplication are continuous.)

For $i \in \N_N$, let $v_i: \Omega_i\times \R \to \R$ be a mapping that is uniformly contractive in the second variable, i.e., there exists an $\ell\in [0,1)$ so that for all $y_1, y_2\in \R$
\be\label{scon}
| v_i(x, y_1) - v_i(x, y_2)| \leq \ell\, | y_1 - y_2|, \quad\forall x\in \Omega.
\ee
Define a \textbf{Read-Bajactarevi\'c (RB) operator} $\Phi: B(\Omega)\to \R^{\Omega}$ by
\be\label{RB}
\Phi f (x) := \sum_{i=1}^N v_i (u_i^{-1} (x), f\circ u_i^{-1} (x))\,\chi_{u_i(\Omega)}(x), 
\ee
where  
$$
\chi_M (x) := \begin{cases} 1, & x\in M\\ 0, & x\notin M\end{cases},
$$
denotes the characteristic function of a set $M$. Note that $\Phi$ is well-defined and since $f$ is bounded and each $v_i$ contractive in the second variable, $\Phi f$ is again an element of $B(\Omega)$.

Moreover, by \eqref{scon}, we obtain for all $f,g\in B(\Omega)$ the following inequality:
\begin{align}\label{estim}
d(\Phi f, \Phi g) & = \sup_{x\in \Omega} |\Phi f (x) - \Phi g (x)|\nonumber\\
& = \sup_{x\in \Omega} |v(u_i^{-1} (x), f_i(u_i^{-1} (x))) - v(u_i^{-1} (x), g_i(u_i^{-1} (x)))|\nonumber\\
& \leq \ell\sup_{x\in \Omega} |f_i\circ u_i^{-1} (x) - g_i \circ u_i^{-1} (x)| \leq \ell\, d(f,g).
\end{align}
To simplify notation, we set $v(x,y):= \sum_{i=1}^N v_i (x, y)\,\chi_{u_i(\Omega)}(x)$ in the above equation. In other words, $\Phi$ is a contraction on the complete metric linear space $B(\Omega)$ and, by the Contraction Mapping Theorem, has therefore a unique fixed point $\ff$ in $B(\Omega)$. This unique fixed point will be called a \textbf{multivariate real-valued fractal function}  (for short, fractal function) and its graph a \textbf{fractal hypersurface of $\R^{n+1}$}.

Next we would like to consider a special choice for the mappings $v_i$. Define $v_i:\Omega\times\R\to \R$ by
\be\label{specialv}
v_i (x,y) := \lambda_i (x) + S_i (x) \,y,\quad i \in \N_N,
\ee
where $\lambda_i \in B(\Omega)$ and $S_i : \Omega\to \R$ is a function. Then $v_i$ given by \eqref{specialv} satisfies condition \eqref{scon} provided that the functions $S_i$ are bounded on $\Omega$ with bounds in $[0,1)$ for then
\begin{align*}
|v_i (x,y_1) - v_i (x,y_2))| & = |S_i (x) \,y_1 - S_i (x) \,y_2| = |S_i(x)|\cdot |y_1 -  y_2|\\
& \leq \|S_i\|_{\infty,\Omega}\, |y_1 - y_2| \leq s\,|y_1 -  y_2|.
\end{align*}
Here, we denoted by $\|\bullet\|_{\infty, \Omega}$ the supremum norm on $\Omega$ and defined 
$$
s := \max\{\|S_i\|_{\infty,\Omega}\st i\in \N_N\}.
$$

Thus, for a fixed set of functions $\{\lambda_1, \ldots, \lambda_N\}$ and $\{S_1, \ldots, S_N\}$, the associated RB operator \eqref{RB} has now the form
\be\label{RB1}
\Phi f = \sum_{i=1}^N \lambda_i\circ u_i^{-1} \,\chi_{u_i(\Omega)} + \sum_{i=1}^N (S_i\circ u_i^{-1})\cdot (f\circ u_i^{-1})\,\chi_{u_i(\Omega)},
\ee
or, equivalently,
\[
\Phi f\circ u_i = \lambda_i + S_i\cdot f,
\]
on $\Omega$ and $\forall\;i\in\N_N$. Thus, we have arrived at the following result.
\begin{theorem}
Let $\Omega$ be an element of $\bH(\R^n)$ and suppose that $\{u_i : \Omega\to \Omega \st i \in \N_N\}$ is a family of bounded bijective mappings satisfying property $\mathrm{(P)}$. Further suppose that the vectors of functions $\blambda := (\lambda_1, \ldots, \lambda_N)$ and $\bS := (S_1, \ldots, S_N)$ are elements of $\underset{i=1}{\overset{N}{\times}} B(\Omega)$. 

Define a mapping $\Phi: \left(\underset{i=1}{\overset{N}{\times}} B(\Omega)\right)\times \left(\underset{i=1}{\overset{N}{\times}} B (\Omega)\right) \times B(\Omega)\to B(\Omega)$ by
\be\label{eq3.4}
\Phi(\blambda)(\bS) f = \sum_{i=1}^N \lambda_i\circ u_i^{-1} \,\chi_{u_i(\Omega} + \sum_{i=1}^N (S_i\circ u_i^{-1})\cdot (f\circ u_i^{-1})\,\chi_{u_i(\Omega)}.
\ee
If $s = \max\{\|S_i\|_{\infty,\Omega}\st i\in \N_N\} < 1$ then the operator $\Phi(\blambda)(\bS)$ is contractive on the complete metric linear space $B(\Omega)$ and its unique fixed point $\ff$ satisfies the self-referential equation
\be\label{3.4}
\ff = \sum_{i=1}^N \lambda_i\circ u_i^{-1} \,\chi_{u_i(\Omega)} + \sum_{i=1}^N (S_i\circ u_i^{-1})\cdot (\ff\circ u_i^{-1})\,\chi_{u_i(\Omega)},
\ee
or, equivalently,
\be
\ff\circ u_i = \lambda_i + S_i\cdot \ff,
\ee
on $\Omega$ and $\forall\;i\in\N_N$.
\end{theorem}

\begin{remark}
Note that the fractal function $\ff:\Omega\to\R$ generated by the RB operator defined by \eqref{eq3.4} does depend on the two $N$-tuples of bounded functions $\blambda, \bS\in \underset{i=1}{\overset{N}{\times}} B (\Omega)$. The fixed point $\ff$ should therefore be written more precisely as $\ff(\blambda)(\bS)$. However, for the sake of notational simplicity, we usually suppress this dependence for both $\ff$ and $\Phi$.
\end{remark}

Now assume that the vector of functions $\bS$ is {\em fixed}. Then the following result found in \cite{GHM} and in more general form in \cite{M97} describes the relationship between the vector of functions $\blambda$ and the fixed point $\ff$ generated by it.
\begin{theorem}\label{thm4}
The mapping $\blambda \mapsto \ff(\blambda)$ defines a linear isomorphism from $\underset{i=1}{\overset{N}{\times}} B(\Omega)$ to $B(\Omega)$.
\end{theorem}
\begin{proof}
Let $\alpha, \beta \in\R$ and let $\blambda, \bmu\in \underset{i=1}{\overset{N}{\times}} B(\Omega)$. Injectivity follows immediately from the fixed point equation \eqref{3.4} and the uniqueness of the fixed point: $\blambda = \bmu$ $\Longleftrightarrow$ $\ff(\blambda) = \ff(\bmu)$.

Linearity follows from \eqref{3.4}, the uniqueness of the fixed point and injectivity: 
\begin{align*}
\ff(\alpha\blambda + \beta \bmu) & = \sum_{i=1}^N (\alpha\lambda_i + \beta \mu_i) \circ u_i^{-1} \,\chi_{u_i(\Omega)}\\
& + \sum_{i=1}^N (S_i\circ u_i^{-1})\cdot (\ff(\alpha\blambda + \beta \bmu)\circ u_i^{-1})\,\chi_{u_i(\Omega)}
\end{align*}
and
\begin{align*}
\alpha \ff(\blambda) + \beta \ff(\bmu) & = \sum_{i=1}^N (\alpha\lambda_i + \beta \mu_i) \circ u_i^{-1} \,\chi_{u_i(\Omega)}\\
& + \sum_{i=1}^N (S_i\circ u_i^{-1})\cdot (\alpha \ff(\blambda) + \beta \ff(\bmu))\circ u_i^{-1})\,\chi_{u_i(\Omega)}.
\end{align*}
Hence, $\ff(\alpha\blambda + \beta \bmu) = \alpha \ff(\blambda) + \beta \ff(\bmu)$.

For surjectivity, we define $\lambda_i := \ff\circ u_i - S_i \cdot \ff$, $i\in \N_N$. Since $\ff\in B(\Omega)$, we have $\blambda\in \underset{i=1}{\overset{N}{\times}} B(\Omega)$. Thus, $\ff(\blambda) = \ff$.
\end{proof}
We will see below that this theorem allows us to define bases for fractal functions. For more details, see \cite{hkm,m95,m10}.
\section{Affinely Generated Fractal Surfaces}\label{sec4}
In this section, we specialize our construction of fractal functions even further. It is our goal to obtain continuous fractal hypersurfaces that are generated by {\em affine} mappings $\lambda_i:\Omega \to \R$ on specially chosen domains $\Omega\subset\R^n$, namely $n$-simplices, and by {\em constant} functions $S_i := s_i\in (-1,1)$.

This type of fractal surface was first systematically introduced in \cite{m90} and generalized in \cite{gh}. Further generalizations were presented in \cite{ghm1,ghm2,hm}. All these constructions are based on using for $\Omega$ certain types of simplicial regions and affine mappings. Later, different types of fractal surface constructions not necessarily based on simplicial regions and affine mappings were published. A short and albeit incomplete list of them is \cite{boud1,boud2,boud3,chand,chand2,feng}. 

In order to set up the connection with wavelet sets, we need to follow the construction that originated in \cite{ghm1,ghm2} and was used in \cite{lm,m10}. To this end, we first need to choose as our domain $\Omega\subset\R^n$ an $n$--simplex.
\begin{definition}
Let $\{p_0, p_1, \ldots, p_n\}$ be a set of affinely independent points in $\R^{n}$. A {\em regular $n$-simplex} in $\R^n$ is defined as the point set
\[
\triangle := \left\{x\in \R^n\;\bigg\vert\; x = \sum_{k=0}^n t_k p_k; \, 0\leq t_k\leq 1;\,\sum_{k=0}^n = 1\right\}.
\]
\end{definition}
\noindent
Over the n--simplex $\triangle$ we consider the following set of functions:
\[
C(\triangle) := C(\triangle, \R) := \{f:\triangle\to\R\st \text{$f$ is continuous on $\triangle$}\}.
\]
It is easy to verify that $(C(\triangle),\|\bullet\|_{\infty,\triangle})$ is a complete metric linear space.

Now let $1<N\in\N$ and suppose that $\{\triangle_i\st i \in\N_N\}$ is a family of nonempty compact subsets of $\triangle$ with the properties that:
\begin{itemize}
\item[(P1)]  $\{\triangle_i\st i \in\N_N\}$ is a set-theoretic partition of $\triangle$ in the sense that
\[
\triangle = \bigcup_{i=1}^N \triangle_i\quad\text{and}\quad \overset{\circ}{\triangle}_i\cap\overset{\circ}{\triangle}_j = \emptyset, \;\forall i\neq j,
\]
\item[(P2)] $\triangle_i$ is similar to $\triangle$, for all $i \in\N_N$. 
\item[(P3)] $\triangle_i$ is congruent to $\triangle_j$, for all $i,j \in\N_N$. 
\end{itemize}
\noindent
Note that conditions (P1), (P2), and (P3) imply the existence of $N$ unique contractive similitudes $u_i:\triangle\to\triangle_i$ given by
\be\label{884}
u_i = a\,O_i + \tau_i, \quad i = 1, \ldots, N,
\ee
where $a < 1$ is the similarity constant or the similarity ratio for $\triangle_i$ with respect to $\triangle$, $O_i$ is an orthogonal transformation on $\R^n$, and $\tau_i$ a translation in $\R^n$.

Let $V$ be the set of vertices of $\triangle$. We denote the set of all distinct vertices of the subtriangles $\{\triangle_i\st i \in\N_N\}$ by $V_i$. Suppose there exists a \textbf{labelling map} $\ell: \cup V_i := \bigcup_{i=1}^N V_i \to V$ so that
\be\label{labelling}
\forall i\in \N_N\;\forall v\in V_i:\quad u_i (\ell(v)) = v.
\ee

Now, suppose that $\{\lambda_i:\triangle\to\R\st i \in\N_N\}$ is a collection of $N$ affine functions and $\{s_i\st i \in\N_N\}$ a set of $N$ real numbers. As in the previous section, we set $\blambda := (\lambda_1, \ldots, \lambda_N)$ and $\bs := (s_1, \ldots, s_N)$. Let us denote by $\A_n := \Aff (\R^n)$ the vector space of all affine mappings $\lambda:\R^n\to\R$. We like to define an RB operator 
\be\label{1000}
\Phi(\blambda)(\bs) f := \sum_{i\in \N_N} (\lambda_i\circ u_i^{-1})\chi_{\triangle_i} + \sum_{i\in \N_N} (s_i\,f\circ u_i^{-1})\chi_{\triangle_i}
\ee
on a subspace $C_0(\triangle)$ of $C(\triangle)$ so that
\[
\Phi(\blambda)(\bs): \left(\underset{i=1}{\overset{N}{\times}} \A_n\right) \times \R^N \times C_0(\triangle)\to C_0(\triangle).
\]

The affine mappings $\lambda_i$ are usually determined by interpolation conditions. Consider thus the interpolation set 
\be\label{Z}
Z := \left\{(v,z_v)\in \triangle\times \R\st v\in\cup V_i\right\}
\ee
Then, for all $i \in\N_N$, the affine mapping $\lambda_i$ is uniquely determined by the interpolation conditions
\be\label{311a}
\lambda_i (\ell(v)) + s_i\,z_{\ell(v)} = z_v, \quad\forall v\in V_i.
\ee

In order for $\Phi f$ to be well-defined on and continuous across adjacent triangles $\triangle_i$ and $\triangle_j$, one needs to impose the following join-up conditions:
\be\label{joinup2d}
\lambda_i\circ u_i^{-1} (x,y) + s_i f\circ u_i^{-1} (x,y) = \lambda_j\circ u_j^{-1} (x,y) + s_j f\circ u_j^{-1} (x,y),
\ee
for all $(x,y)\in e_{ij} := \triangle_i\cap\triangle_j$, $i,j\in \N_N$ with $i\neq j$. The quantity $e_{ij}$ is called a \textit{common edge} of $\triangle_i$ and $\triangle_j$. 

The next result, which is essentially Theorem 6 in \cite{ghm1} gives conditions for \eqref{joinup2d} to be satisfied.

\begin{theorem}\label{thm5}
Let $\triangle$ be a $n$--simplex and $\{\triangle_i\st i \in\N_N\}$ a family of nonempty compact subsets of $\triangle$ satisfying conditions (P1), (P2), and (P3). Let $Z$ be an interpolation set of the form \eqref{Z} and let $C_0 (\triangle) := \left\{f\in C(\triangle)\st f(v) = z_v, \;v\in\cup V_i\right\}$. Suppose that there exists a labelling map $\ell$ as defined in \eqref{labelling}. Further suppose that $\bs := (s, \ldots, s)$, with $|s| < 1$. Then the RB operator $\Phi$ defined by \eqref{1000} maps $C_0(\triangle)$ into itself, is well-defined, and contractive on the complete {\em metric subspace} $(C_0(\triangle), \|\bullet\|_{\infty,\triangle})$ of $(C(\triangle), \|\bullet\|_{\infty,\triangle})$.
\end{theorem}

The unique fixed point of the RB operator in Theorem \ref{thm5} is called a \textbf{multivariate real-valued affine fractal function} and its graph an \textbf{affinely generated fractal hypersurface} or an \textbf{affine fractal hypersurface}.

\begin{example}\label{ex2}
Let $n:=2$ and suppose we are given the $2$--simplex and its associated partition as depicted in Figure \ref{simp} below.
\begin{figure}[h!]\label{simp}
\begin{center}
\includegraphics[width=5cm,height=4cm]{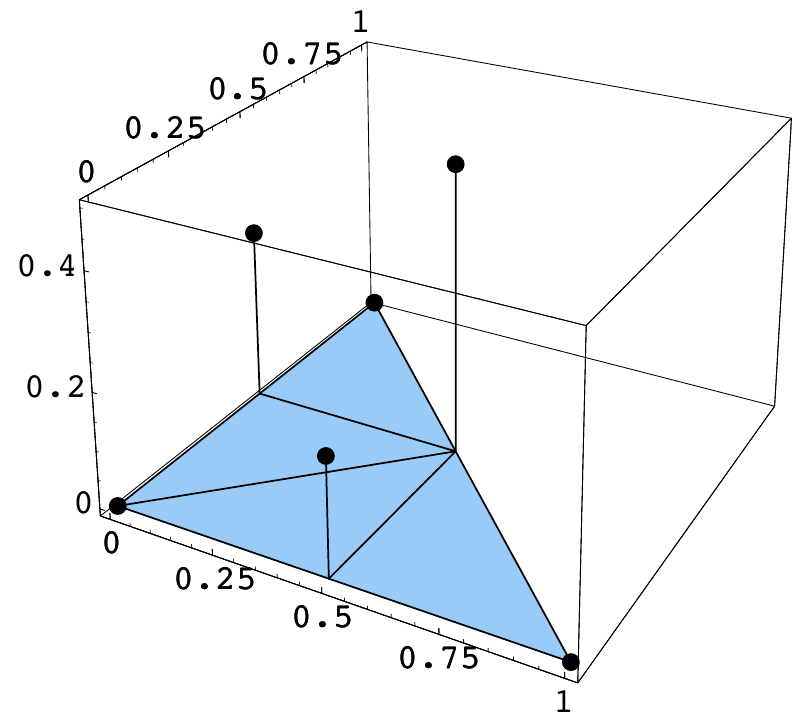}
\caption{A $2$--simplex and its associated partition.}
\end{center}
\end{figure}

\noindent
A simple computation yields for the four mappings $u_i$ the following expressions:
\begin{align*}
u_1\begin{pmatrix}x\\y\end{pmatrix} &= \begin{pmatrix} \frac12 & 0 \\0 & \frac12\end{pmatrix}\begin{pmatrix}x\\y\end{pmatrix} + \begin{pmatrix}1/2\\0\end{pmatrix},&
u_2\begin{pmatrix}x\\y\end{pmatrix} & = \begin{pmatrix} -\frac12 & 0\\0 & \frac12\end{pmatrix}\begin{pmatrix}x\\y\end{pmatrix} + \begin{pmatrix}1/2\\0\\\end{pmatrix}\\
u_3\begin{pmatrix}x\\y\end{pmatrix} &= \begin{pmatrix} \frac12 & 0\\0 & -\frac12\end{pmatrix}\begin{pmatrix}x\\y\end{pmatrix} + \begin{pmatrix}0\\1/2\end{pmatrix},&
u_4\begin{pmatrix}x\\y\end{pmatrix} & = \begin{pmatrix} \frac12 & 0\\0 & \frac12\end{pmatrix}\begin{pmatrix}x\\y\end{pmatrix} + \begin{pmatrix}0\\1/2\end{pmatrix}.
\end{align*}
The affine functions $\lambda_i$ are then given by 
\begin{align*}
\lambda_1 (x, y) & = \lambda_2 (x, y) = -z_1 x + (z_2 - z_1) y + z_1,\\
\lambda_3 (x, y) & = \lambda_4 (x, y) = (z_2 - z_3) x - z_3 y + z_3, 
\end{align*}
where $z_1, z_2,$, and $z_3$ are the nonzero $z$ values. (See the figure above.)
The sequence of graphs in Figure \ref{surf} shows the generation of the fractal surface.
\begin{figure}[h!]
\begin{center}\includegraphics[width=5cm,height=4cm]{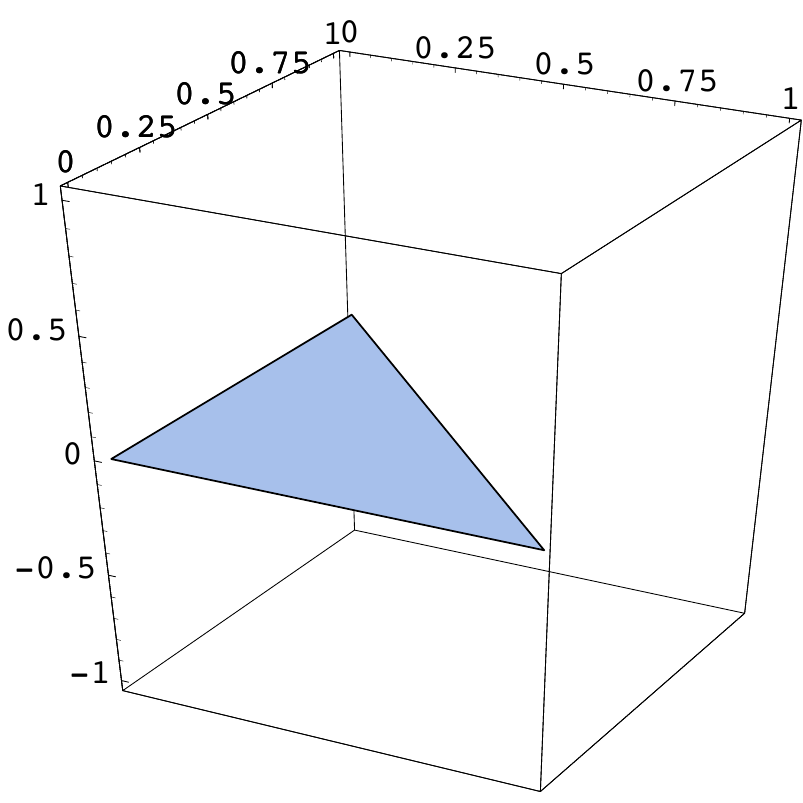}\qquad
\includegraphics[width=5cm,height=4cm]{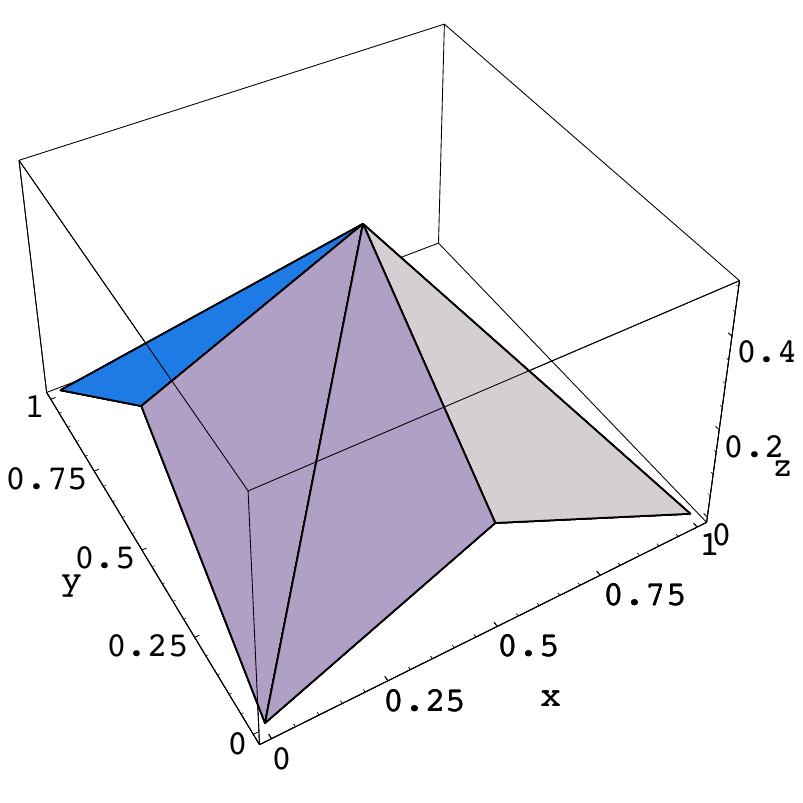}
\end{center}
\begin{center}\includegraphics[width=5cm,height=4cm]{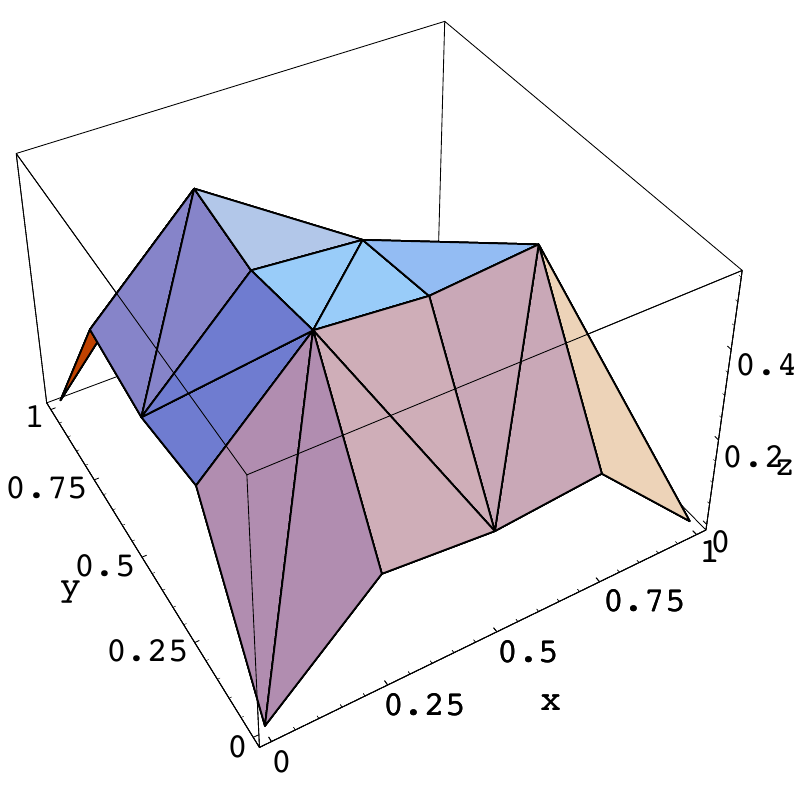}\qquad
\includegraphics[width=5cm,height=4cm]{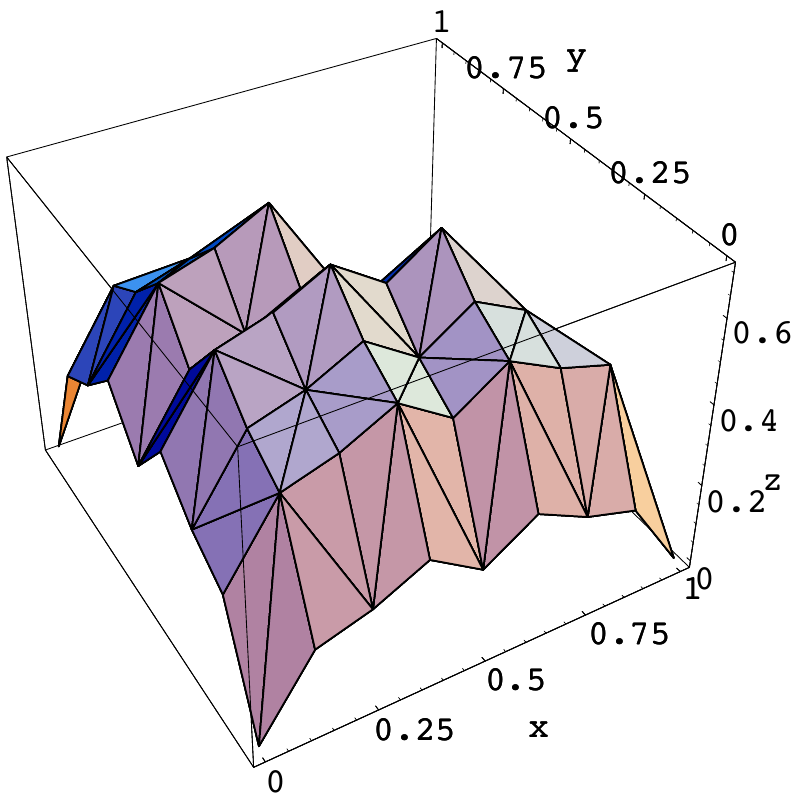}
\end{center}%
\begin{center}\includegraphics[width=5cm,height=4cm]{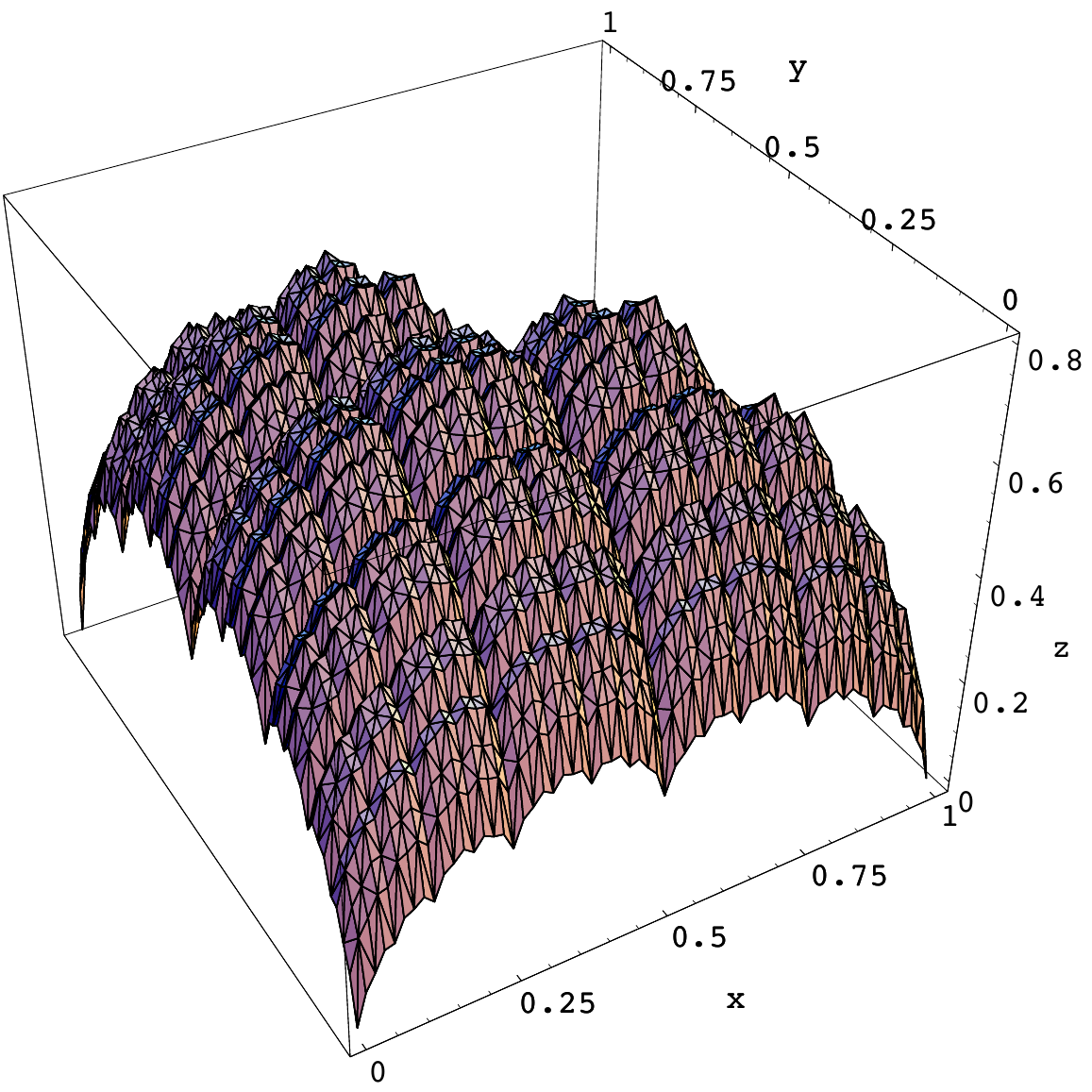}
\end{center}
\caption{The generation of an affine fractal surface.}\label{surf}
\end{figure}
\end{example}

Using Theorem \ref{thm5} we can construct a basis for fractal surfaces. For this purpose, we denote by $\cS(\triangle; \A_n)$ the set of all fractal functions generated by the RB operator \eqref{1000} subject to the conditions stated in Theorem \ref{thm5}.  The set $\cS(\triangle; \A_n)$ becomes for {\em fixed} $\bs$ a complete metric linear space inheriting its metric from $C(\triangle)$. Note that $\dim\cS(\triangle; \A_n) = 3N - \card \left(\cup V_i\right)$; 3 free parameter for each of the $N$ affine functions $\lambda_i$ and $\card \left(\cup V_i\right)$ many interpolation conditions at the vertices. 

This observation suggests the construction of a $n$-dimensional system of Lagrange interpolants for each affine fractal function $\ff$ of the form
\be\label{444}
\fB := \left\{\fb_v:\triangle\to\R\st v\in \cup V_i\right \},
\ee
where each $\fb_v$ is the unique affine fractal hypersurface interpolating the set 
$$
Z_v := \left\{(v',\delta_{v v'})\st v'\in\cup V_i\right \}, \qquad v\in\cup V_i.
$$ 
We also refer to the set of these affine fractal functions as an \textbf{affine fractal basis}. Thus, we have the following result.
\begin{theorem}
Let $\ff$ be an affine fractal function with associated interpolation set $Z$ as defined in \eqref{Z}. Then there exist an affine fractal basis of Lagrange-type \eqref{444} of cardinality $\card \left(\cup V_i\right)$ so that
\[
\ff = \sum_{v\in\cup V_i} z_v\,\fb_v.
\]
\end{theorem}
\begin{example}
In Figure \ref{basfun}, three of the six graphs of the affine fractal basis functions for the affine fractal surface constructed in Example \ref{ex2} are displayed.
\begin{figure}[h!]
\begin{center}
\includegraphics[width=5cm,height=4cm]{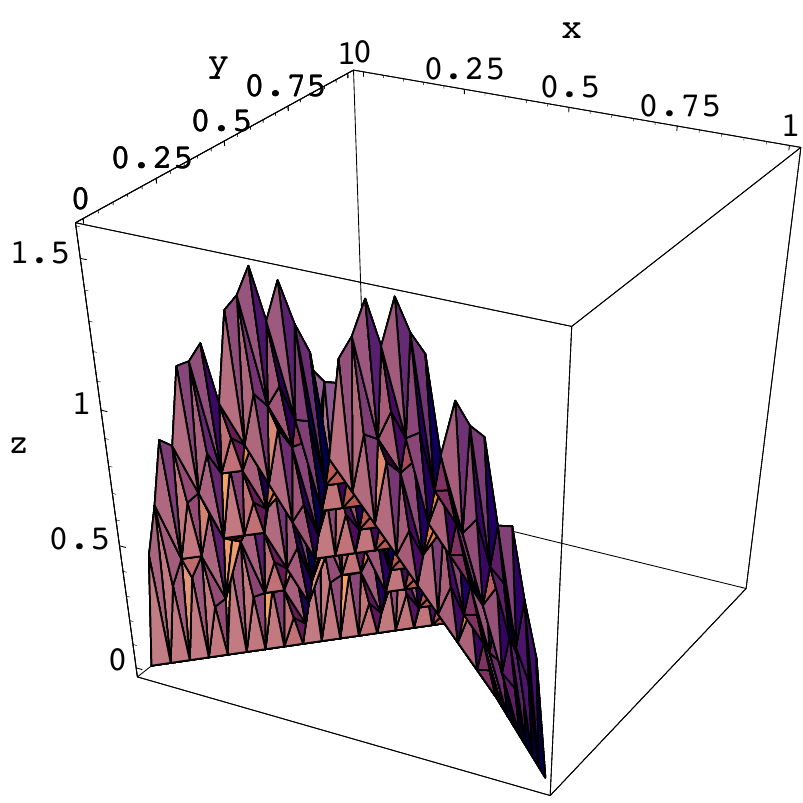}\hspace{1cm}
\includegraphics[width=5cm,height=4cm]{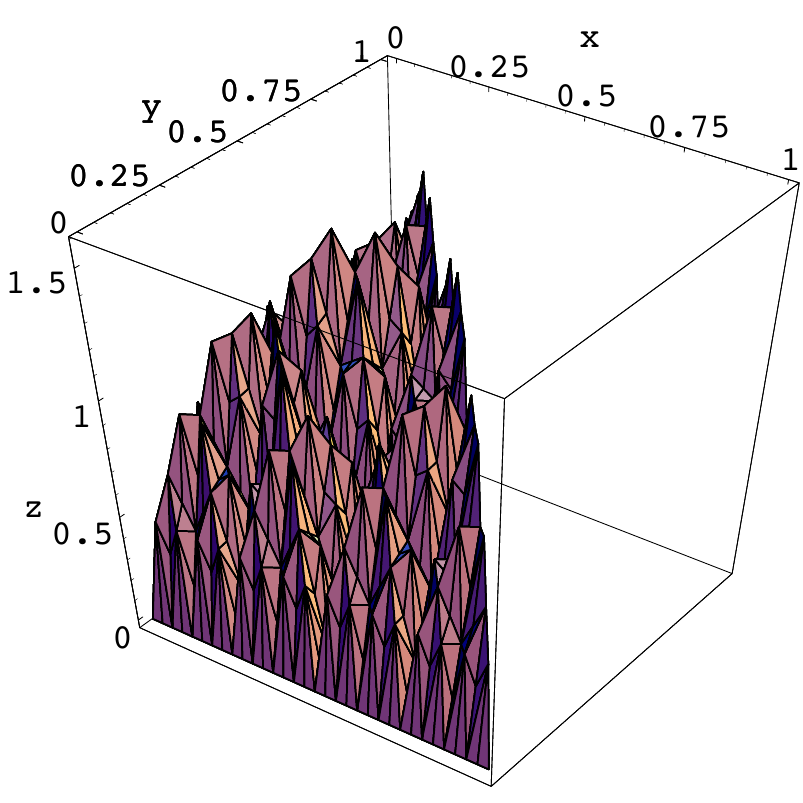}%
\end{center}
\begin{center}\includegraphics[width=5cm,height=4cm]{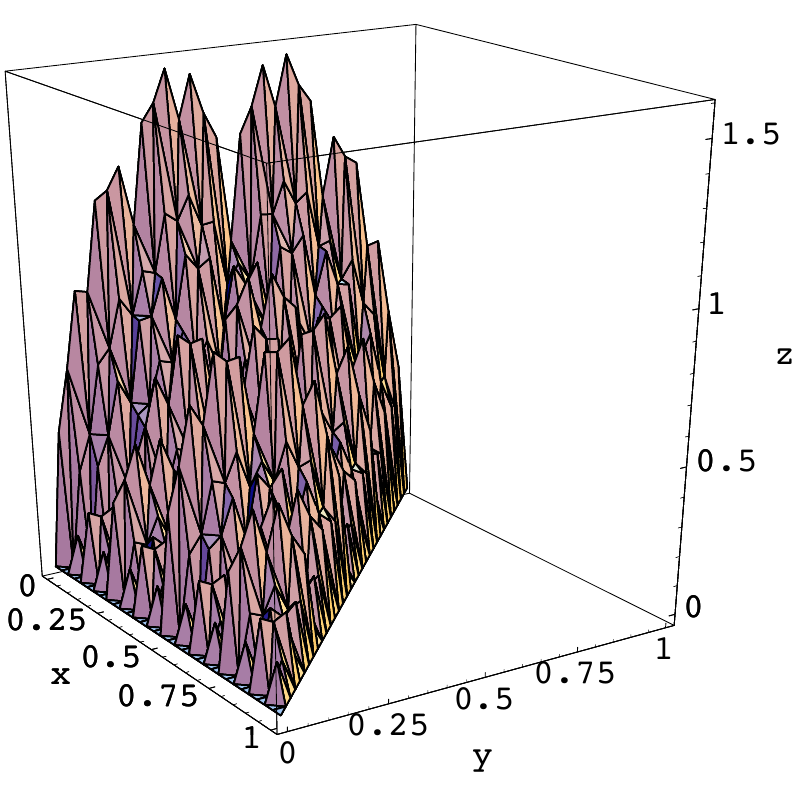}
\end{center}%
\caption{Some affine fractal basis functions.}\label{basfun}
\end{figure}
\end{example}

As $\cS(\triangle; \A_n)$ is finite-dimensional, we can apply the Gram--Schmidt Orthogonalization Process to obtain an {\em orthonormal basis} for $\cS(\triangle; \A_n)$ consisting of affine fractal basis functions. (See also \cite{ghm1,ghm2,m95,m10}.) This orthonormal basis will play an important role in the connection with wavelet sets and affine Weyl groups.
\section{Root Systems and Affine Weyl Groups}\label{sec5}
In the current section, we introduce root systems, reflections about affine hyperplane and the associated affine Weyl groups. We only present those concepts that are of importance for later developments and refer the reader to \cite{bour,brown,cox,gb,gun,hum,lang} for a more in-depth presentation of reflection groups and root systems.

For the remainder of this section, we denote by $\E^n$ $n$-dimensional Euclidean space endowed with the Euclidean inner product $\inn{\bullet}{\bullet}$.

Let $H\subset\mathbb{E}^n$ be a linear hyperplane, i.e., a codimension 1 linear subspace of $\E^n$. 
\begin{definition}\label{def5}
A linear transformation $r:\E^n\to\E^n$ is called a \textbf{reflection about $H$} if 
\begin{itemize}
\item[(i)] $r (H) = H$;
\item[(ii)] $r(x) = -x$, for all $x\in H^\perp$.
\end{itemize}
\end{definition}
\noindent
A straight-forward computation yields an explicit representation of $r$:
\be\label{r}
r_\alpha(x) = x - \displaystyle{\frac{2\inn{x}{\alpha}}{\inn{\alpha}{\alpha}}}\,\alpha,\quad\text{for fixed $0\neq \alpha\in H^\perp$.}
\ee
A discrete group generated by a set of reflections in $\E^n$ is called a \textbf{Eulidean reflection group}. An abstract group that has a representation in terms of reflections is termed a \textbf{Coxeter group}. More precisely, a Coxeter group is a discrete group $\mathcal{C}$ with a finite set of generators $\{\alpha_i\,:\,i = 1,\ldots, k\}$ satisfying
        \[
        \mathcal{C} := \bigl\langle \alpha_1,\ldots, \alpha_k\,|\, (\alpha_i \alpha_j)^{m_{ij}} = 1,\; 1 \leq i,j \leq k\bigr\rangle
        \]
        where $m_{ii} = 1$, for all $i$, and $m_{ij}\geq 2$, for all $i\neq j$. ($m_{ij}=\infty$ is used to indicate that no relation exists.)
It can be shown that finite Coxeter groups are isomorphic to finite Euclidean reflection groups. \cite{cox}
\begin{example} 
Klein's Four-Group $V$ or dihedral group $D_4$ of order 4: 
$$
V = D_4 = \bigl\langle \alpha,\beta \st \alpha^2 = \beta^2 = (\alpha \beta)^2 = 1\bigr\rangle.
$$
Geometrically: Symmetry group of the unit square $[-\frac12,\frac12]\times[-\frac12,\frac12]$ centered at the origin.
\end{example}

\begin{definition}
A \textbf{root system} ${R}$ is a finite set of vectors $\alpha_1, \ldots, \alpha_k\in$ $\mathbb{R}^n\setminus\{0\}$ with the properties that
\begin{enumerate}
\item[(R1)]   $\mathbb{R}^n = \textrm{span}\,\{\alpha_1,\ldots, \alpha_k\}$;
\item[(R2)]   $\alpha, t \alpha\in R$ if and only if $t = \pm 1$;
\item[(R3)]   $\forall \alpha\in R$: $r_\alpha(R)= R$, where $r_\alpha$ is the reflection through the hyperplane orthogonal to $\alpha$;
\item[(R4)]   $\forall \alpha, \beta\in R$: $\displaystyle{\inn{\beta}{\alpha^\vee}\in\mathbb{Z}}$, where $\alpha^\vee := 2\,\alpha/\inn{\alpha}{\alpha}$.
\end{enumerate}
The dimension of $\R^n$, i.e., $n$, is called the \textbf{rank} of the root system.
\end{definition}
\noindent
We note that condition (R4) restricts the possible angles between roots: Let $\alpha, \beta$ be two roots and let
\[
n(\beta,\alpha) := \inn{\beta}{\alpha^\vee} = 2 \frac{\inn{\alpha}{\beta}}{\inn{\alpha}{\alpha}} \in \Z.
\]
Denote by $|\alpha| := \inn{\alpha}{\alpha}^{1/2}$ the length of $\alpha$ and by $\theta$ the angle between $\alpha$ and $\beta$. Then $\inn{\alpha}{\beta} = |\alpha|\,|\beta|\,\cos\theta$ and thus
\be\label{n}
n(\beta,\alpha) = 2\,\frac{|\beta|}{|\alpha|}\,\cos\theta.
\ee
It follows from $\eqref{n}$ that 
$$
n(\beta,\alpha)\, n(\alpha,\beta) = 4 \cos^2\theta.
$$
Hence,
$$
n(\beta,\alpha)\in \Z\quad\Longrightarrow\quad 4\cos^2\theta \in\{0,1,2,3,4\}.
$$
The following table shows the possible angles $\theta$ and thus the relation between the two roots $\alpha$ and $\beta$.
\nl
\begin{center}
\hspace*{-19pt}
\begin{tabular}{lcrlcrlcrlcr}
$n(\beta,\alpha)$ & = & 0, & $n(\alpha,\beta)$ & = & 0, & $\theta$ & = & $\pi/2$.\\
$n(\beta,\alpha)$ & = & 1, & $n(\alpha,\beta)$ & = & 1, & $\theta$ & = & $\pi/3$, & $|\beta|$ & = & $|\alpha|.$\\
$n(\beta,\alpha)$ & = & -1, & $n(\alpha,\beta)$ & = & -1, & $\theta$ & =& $2\pi/3$, & $|\beta|$ & = & $|\alpha|.$\\
$n(\beta,\alpha)$ & = & 1, & $n(\alpha,\beta)$ & = & 2, & $\theta$ & =& $\pi/4$, & $|\beta|$ & = & $\sqrt{2}|\alpha|.$\\
$n(\beta,\alpha)$ & = & -1, & $n(\alpha,\beta)$ & = & -2, & $\theta$ & =& $3\pi/4$, & $|\beta|$ & = & $\sqrt{2}|\alpha|.$\\
$n(\beta,\alpha)$ & = & 1, & $n(\alpha,\beta)$ & = & 3, & $\theta$ & =& $\pi/6$, & $|\beta|$ & = & $\sqrt{3}|\alpha|.$\\
$n(\beta,\alpha)$ & = & -1, & $n(\alpha,\beta)$ & = & -3, & $\theta$ & =& $5\pi/6$, & $|\beta|$ & = & $\sqrt{3}|\alpha|.$\\
\end{tabular}
\end{center}
\nl
Note that the value of $ \theta$ determines both $n(\alpha,\beta)$ as well as $\{|\alpha|/|\beta|, |\beta|/|\alpha|\}$.
\begin{example}
The following are three examples of root systems in $\E^2$.
\begin{center}
\begin{figure}[h!]
\includegraphics[width = 3.5cm, height = 3cm]{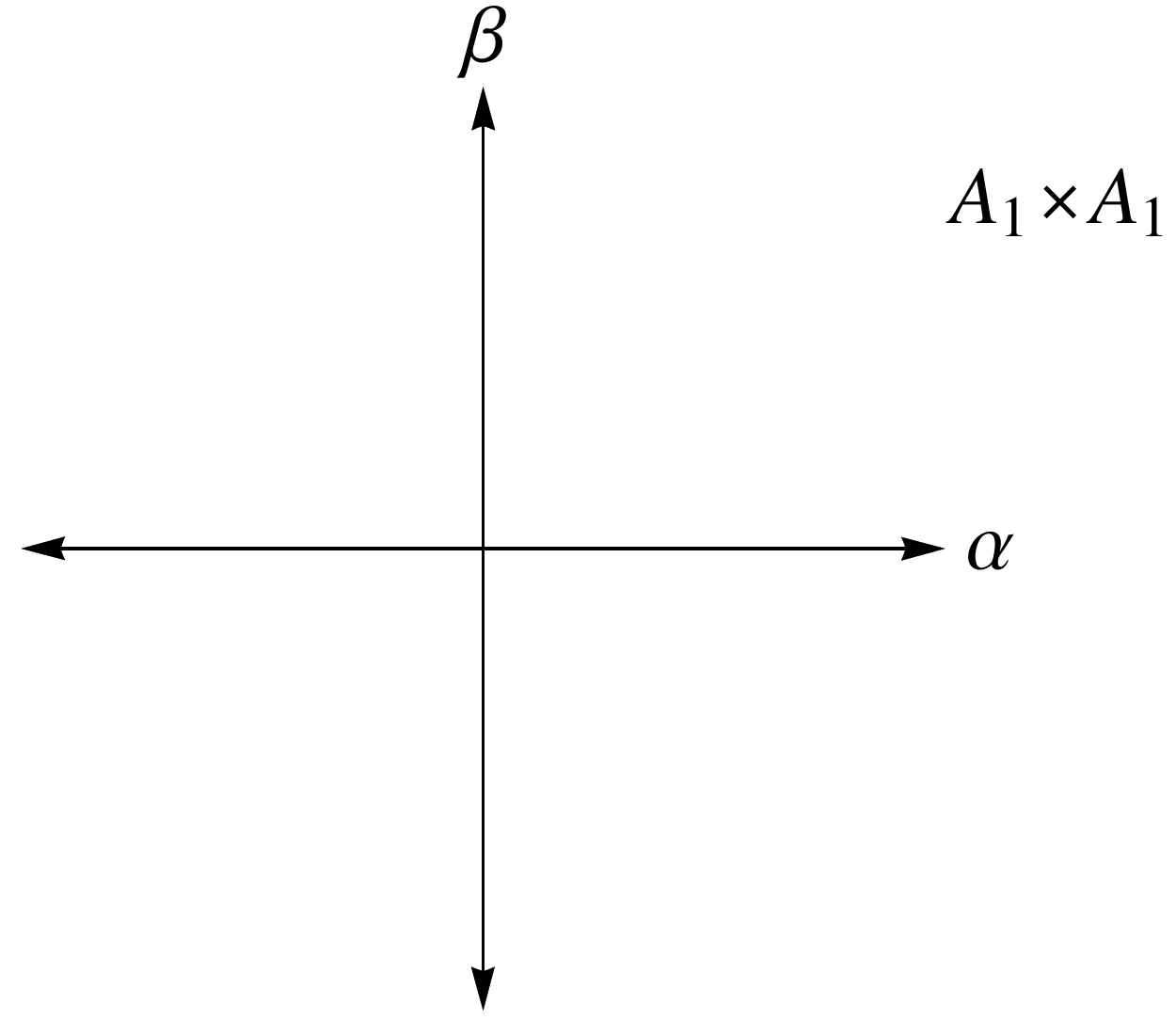}
\hspace*{2cm}\includegraphics[width = 3.5cm, height = 3cm]{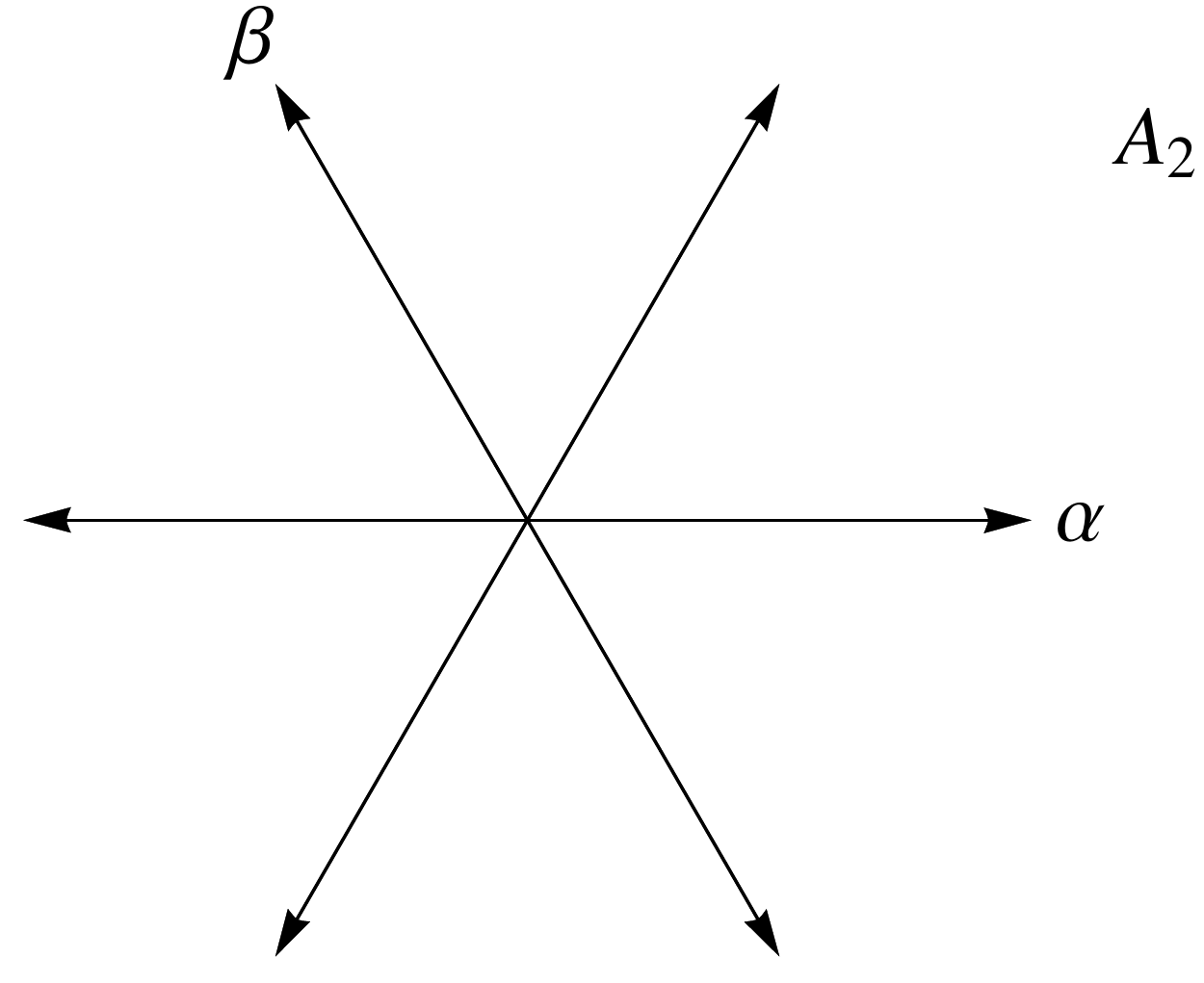}\vspace*{1cm}
\includegraphics[width = 3.5cm, height = 3cm]{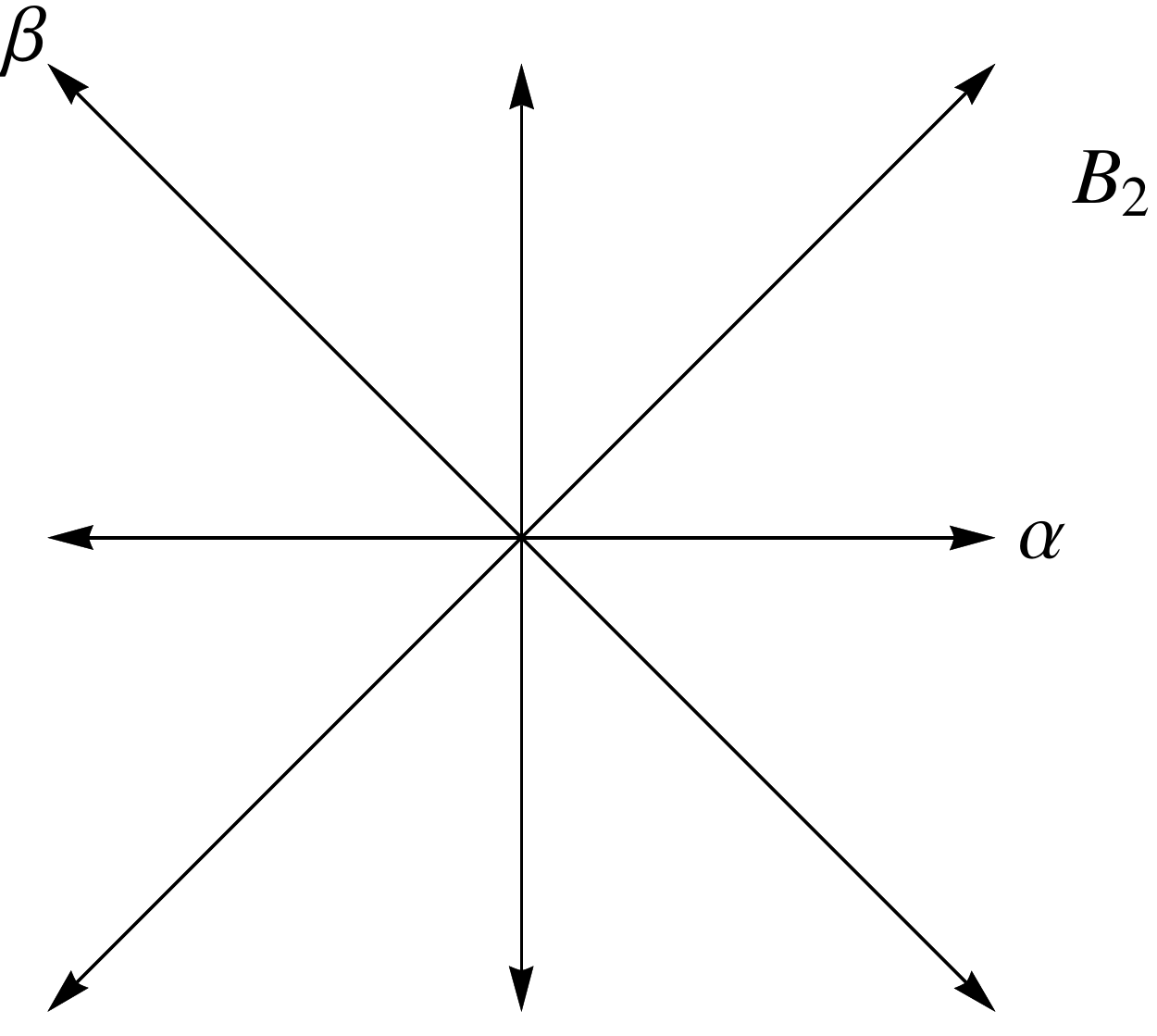}
\caption{Three root systems of rank 2.}
\end{figure}
\end{center}
\end{example}
Roots may be divided into two classes. To this end, choose $v\in \R^n$ so that $\inn{v}{\alpha} \neq 0$, for all roots $\alpha\in R$.
\ml   
\noindent
\textbf{Set of positive roots}: $\quad R^+ :=\{\alpha\in R\st \inn{v}{\alpha} > 0\}$.
\ml  
\noindent 
\textbf{Set of negative roots}: $\hspace*{8pt} R^- :=\{\alpha\in R\st \inn{v}{\alpha} < 0\}$.
\ml
\noindent
Clearly, $R = R^+ \;\dot\cup\; R^-$, where $\dot\cup$ denotes disjoint union.

A root in ${R^+}$  is called \textbf{simple} if it can not be written as the sum of two elements of $R^+$. The set $\Delta \subset R^+$ of simple roots forms a basis of $\R^n$ with the property that every $\alpha\in R$ can be written in the form
\[
\alpha = \sum k_i \,\delta_i, \quad k_i\in \Z, \; \delta_i\in \Delta,
\]
where \textit{all} $k_i > 0$ or \textit{all} $k_i < 0$.

We summarize some properties of roots and root systems below. 
\begin{proposition}
Let $R$ be a root system and $R^+$ the set of positive roots.
\begin{enumerate}
\item $R^+$ is that subset of $R$ for which the following two conditions hold:
\begin{itemize}
\item[(i)] Given $\alpha\in R$, then either $\alpha\in R^+$ or $-\alpha\in R^+$;
\item[(ii)] For all pairs $(\alpha,\beta)\in R^+\times R^+$ with $\alpha\neq\beta$, such that $\alpha+\beta\in R$, $\alpha+\beta\in R^+$.
\end{itemize}
\item For $\alpha, \beta\in \Delta$ with $\alpha\neq\beta$, $\inn{\alpha}{\beta} \leq 0$. (I.e., the angle between two simple roots is always obtuse.)
\item Assume that $\alpha,\beta\in R$, $\alpha$ and $\beta$ are not multiples of each other, and $\inn{\alpha}{\beta} > 0$. Then $\alpha - \beta\in R$.
\end{enumerate}
\end{proposition}

Now suppose that $\alpha,\beta\in R$. Let $H$ be the hyperplane orthogonal to $\alpha$. We denote by $\alpha^*\in (\E^n)^*$ the unique element in the algebraic dual of $\E^n$ such that 
\begin{center}
$\alpha^* (H) = 0$,\quad\text{and}\quad $\alpha^*(\alpha) = 2.$
\end{center}
Then we may rewrite \eqref{r} in the form
\[
r_\alpha (\beta) = \beta - \alpha^*(\beta) \alpha = (1 - \alpha^*\otimes\alpha)(\beta).
\]
Since $\E^n$ and $(\E^n)*$ are isomorphic via the Euclidean inner product $\inn{\bullet}{\bullet}$, there exists a unique element $\alpha^\vee\in \E^n$ so that $\alpha^* = \inn{\alpha^\vee}{\bullet}$. $\alpha^\vee$ is called the \textbf{coroot} of $\alpha$. The lattice in $\E^n$ spanned by the roots $R$, respectively, coroots $R^\vee$ is called the \textbf{root}, respectively, \textbf{coroot lattice}.

The finite family of hyperplanes $\{H_\alpha \st \alpha\in R\}$, where $H_\alpha$ is the hyperplane orthogonal to $\alpha$, partitions $\E^n$ into finitely many regions. The connected components of $\E^n \setminus \displaystyle{\bigcup_{\alpha\in R} H_\alpha}$ are called the (open) \textbf{Weyl chambers} of $\E^n$.
	  
The \textbf{fundamental Weyl chamber $C$ relative to} $\Delta$ is defined by
\[
C := \bigcap_{\delta\in \Delta} \{v\in \E^n\st \inn{v}{\delta} > 0\}.
\]
It should be clear that $C$ is a simplicial cone, hence convex and connected.
\begin{center}
\begin{figure}[h!]
\includegraphics[width=5cm,height=4cm]{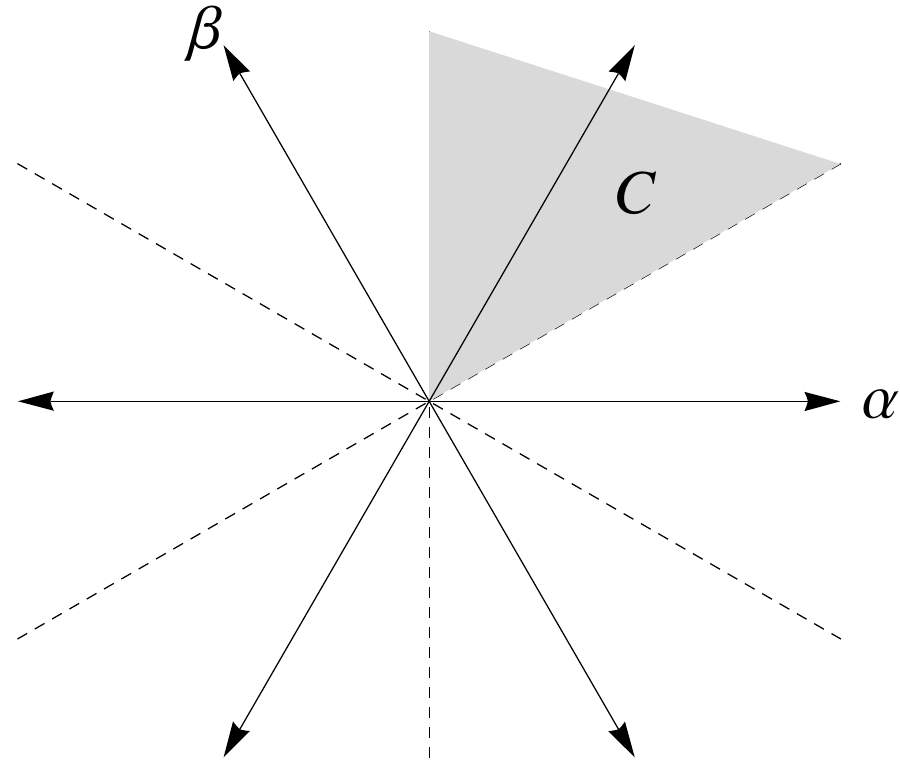}
\caption{A root system, its Weyl chambers (regions between dashed lines) and the fundamental Weyl chamber $C$.}
\end{figure}
\end{center}
\begin{definition}
The subgroup of the isometry group of a root system that is generated by the reflections through the hyperplanes orthogonal to the roots is called the \textbf{Weyl group $\mathcal{W}$} of the root system $R$.
\end{definition}

As the root system is finite, $\mathcal{W}$ is a finite reflection group. Moreover, $\mathcal{W}$ acts simply transitive on the Weyl chambers, i.e., if $C_1$ and $C_2$ are two Weyl chambers and $x_j\in C_j$, $j=1,2$, then there exists a unique $r\in \cW$ so that $r(x_1) = x_2$.

An \textbf{affine hyperplane} with respect to a root system $R$ is defined by
\[
H_{\alpha,k} := \{x\in\mathbb{R}^n\st\inn{x}{\alpha} = k\},\quad \alpha\in  R,\;k\in\mathbb{Z}.
\]
We can also consider reflections $r_{\alpha,k}$ about affine hyperplanes. Employing conditions (i) and (ii) in Definition \ref{def5} applied now to affine hyperplanes, we obtain the following expression for such reflections:
\be\label{refl}
r_{\alpha,k}(x) = x - \displaystyle{\frac{2(\inn{x}{\alpha}-k)}{\inn{\alpha}{\alpha}}}\,\alpha =: r_\alpha (x) + k\,{\alpha}^\vee.
\ee
\begin{definition}
Let $R$ be a root system and $\{H_{\alpha,k}\st \alpha\in R, \,k\in \Z\}$ its system of affine hyperplanes. The \textbf{affine Weyl group} $\widetilde{\mathcal{W}}$ for $R$ is the infinite group generated by the reflections $r_{\alpha,k}$ about the affine hyperplanes $H_{\alpha,k}$:
\[
\widetilde{\mathcal{W}} := \bigl\langle r_{\alpha,k}\st \alpha\in R,\, k\in\mathbb{Z}\bigr\rangle
\]
\end{definition}
The next result characterizes the affine Weyl group of a root system and relates it to the finite Weyl group and the lattice generated by the coroots.
\begin{theorem}[Bourbaki]
The affine Weyl group $\widetilde{\mathcal{W}}$ of a root system $R$ is the semi-direct product\footnote{Let $(G,\cdot)$ be a group with identity element $e$. Suppose $H$ is a subgroup of $G$ and $N$ a normal subgroup of $G$. Further suppose that $G = H\cdot N$ and $H\cap N = \{e\}$. Then $G$ is called the \textbf{semi-direct product of $H$ and $N$}.} $\mathcal{W}\ltimes \Gamma$, where $\Gamma$
is the abelian group generated by the coroots ${\alpha}^\vee$. Moreover, $\Gamma$ is the subgroup of translations of $\widetilde{\mathcal{W}}$ and
$\mathcal{W}$ the isotropy group (stabilizer) of the origin. The group $\mathcal{W}$ is finite and $\Gamma$ infinite.
\end{theorem}
In this context, also note the particular form of an affine reflection \eqref{refl}; it is the sum of a reflection across a linear hyperplane plus a translation along the lattice spanned by the coroot.

It can be shown that $\widetilde{\mathcal{W}}$ has a \textbf{fundamental domain $C\subset\E^n$} in the sense that no $r\in\widetilde{\mathcal{W}}$ maps a point of $C$ to another point of $C$, and for all $x\in\E^n$ there exists an $r\in\widetilde{\mathcal{W}}$ such that $r(x)\in C$. Furthermore, $C$ is a compact and convex simplex.

All affine Weyl groups (and therefore their fundamental domains) are classified. For a given dimension $n\in \N$ there exists only a \textit{finite} number of possible groups  and thus fundamental domains. The classification of so-called \textbf{irreducible} root systems (root systems that cannot be written as the union of two root systems $R_1$ and $R_2$ such that $\inn{\alpha_1}{\alpha_2} = 0$, for $\alpha_i\in R_i$, $i=1,2$) follows from the representation theory of simple Lie algebras. 

The classification yields four infinite families $A_n$ ($n\geq 1$), $B_n$ ($n\geq 2$), $C_n$ ($n\geq 3$), and $D_n$ ($n\geq 4$) and five exceptional cases $E_6$, $E_7$, $E_8$, $F_4$, and $G_2$. (The subscript indicates the rank of the root system.) For the cardinality of these families and the explicit construction of their root system as well as the geometric description of the fundamental domains, we refer the reader to \cite{bour,hum}.

Figure \ref{root} shows the three irreducible root systems of rank 2, their fundamental domains $C$ and coroot lattices. Note that the fundamental domains consist of an equiangular triangle (root system $A_2$), a $90^\circ$--$45^\circ$--$45^\circ$ triangle ($B_2$), and a $90^\circ$--$30^\circ$--$60^\circ$ triangle ($G_2$).
\begin{figure}[h!]
\begin{center}
\includegraphics[width=3cm,height=3cm]{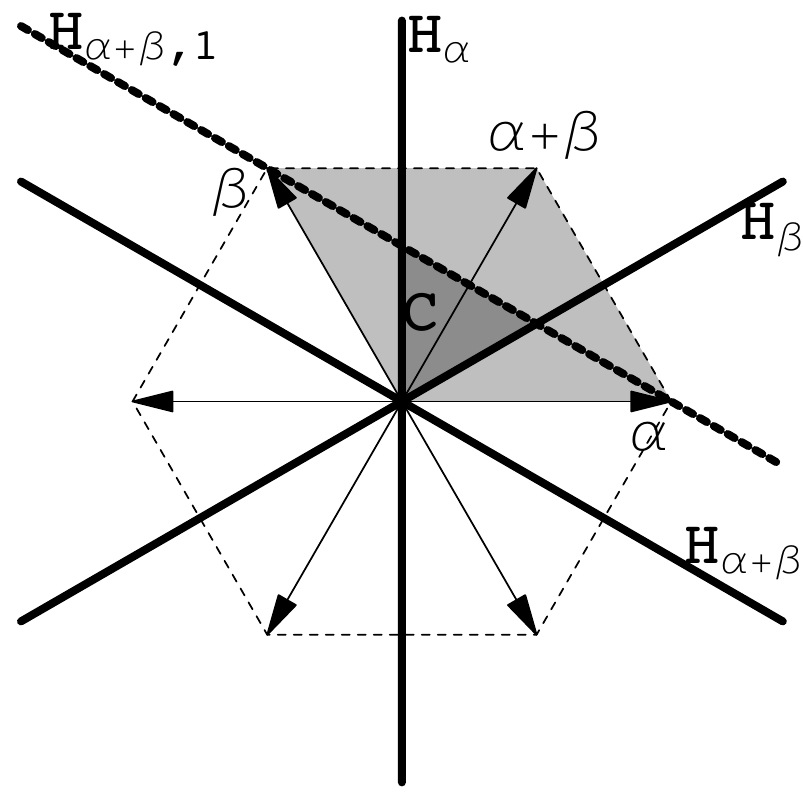}\hspace{2cm}
\includegraphics[width=3cm,height=3cm]{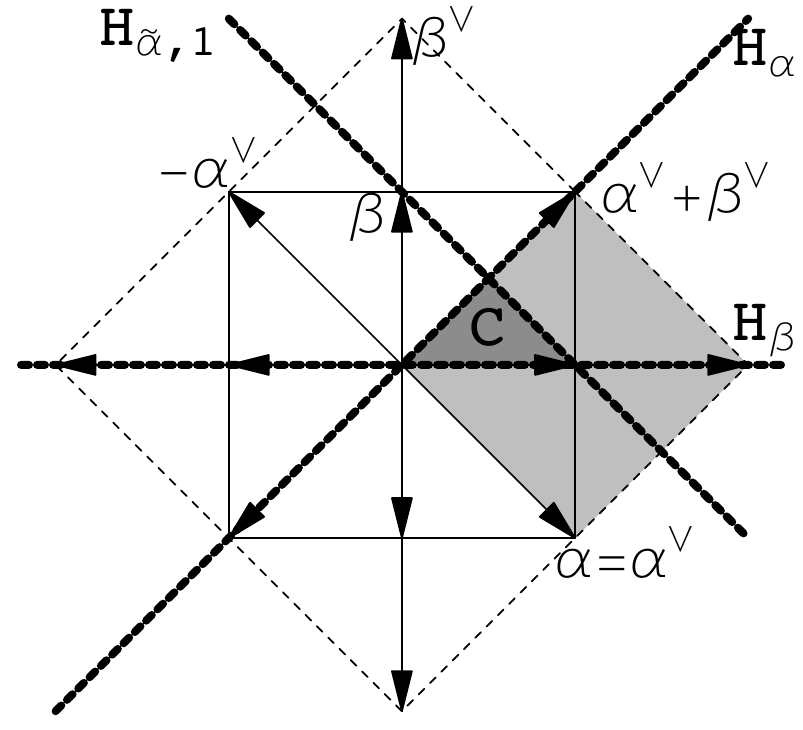}\\
\includegraphics[width=3cm,height=3cm]{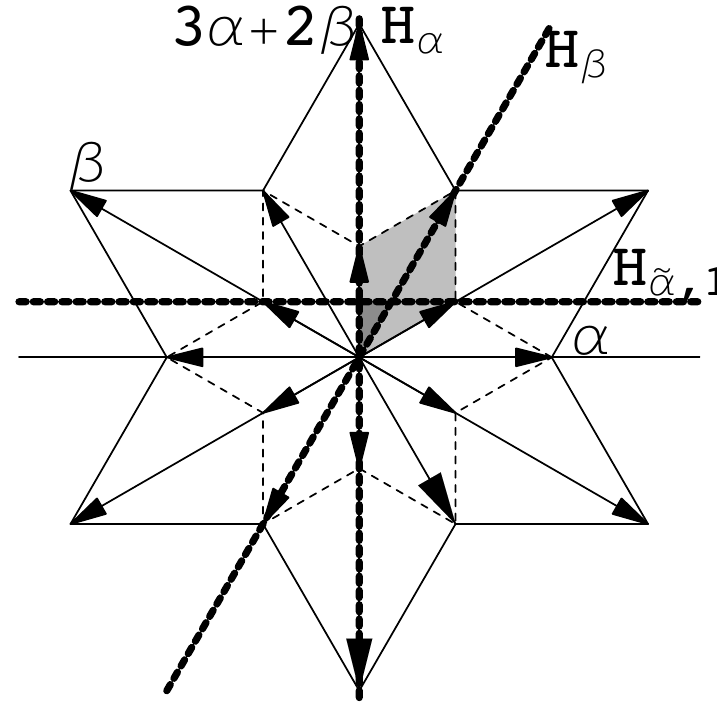}
\end{center}
\caption{The root systems $A_2$, $B_2$, and $G_2$.}\label{root}
\end{figure}

The final concept we need to introduce to set up the connection between affine fractal surfaces, affine Weyl groups and wavelet sets is that of foldable figure.
\begin{definition}
A compact connected subset $F$ of $\mathbb{E}^n$ is called a \textbf{foldable figure} if and only if there exists a finite set $\mathcal{S}$ of affine hyperplanes that cuts $F$ into finitely many congruent subfigures $F_1, \ldots, F_m$, each similar to $F$, so that reflection in any of the cutting
hyperplanes in $\mathcal{S}$ bounding $F_k$ takes it into some $F_\ell$.
\end{definition}

\begin{figure}[h]
\begin{center}
\includegraphics[width=1.5cm,height=1.5cm]{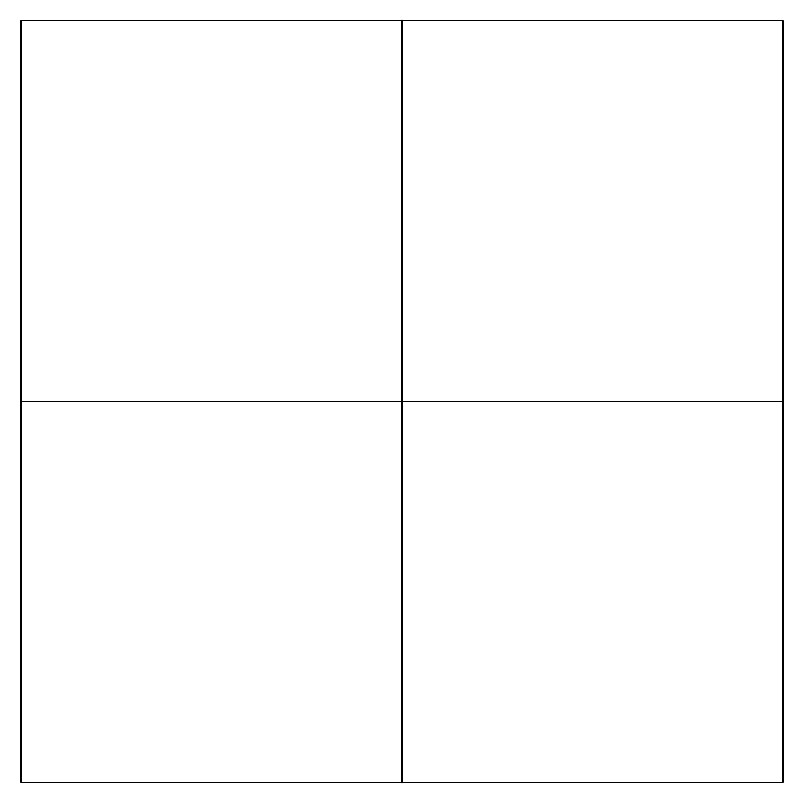}\hspace{2cm}
\includegraphics[width=1.5cm,height=1.5cm]{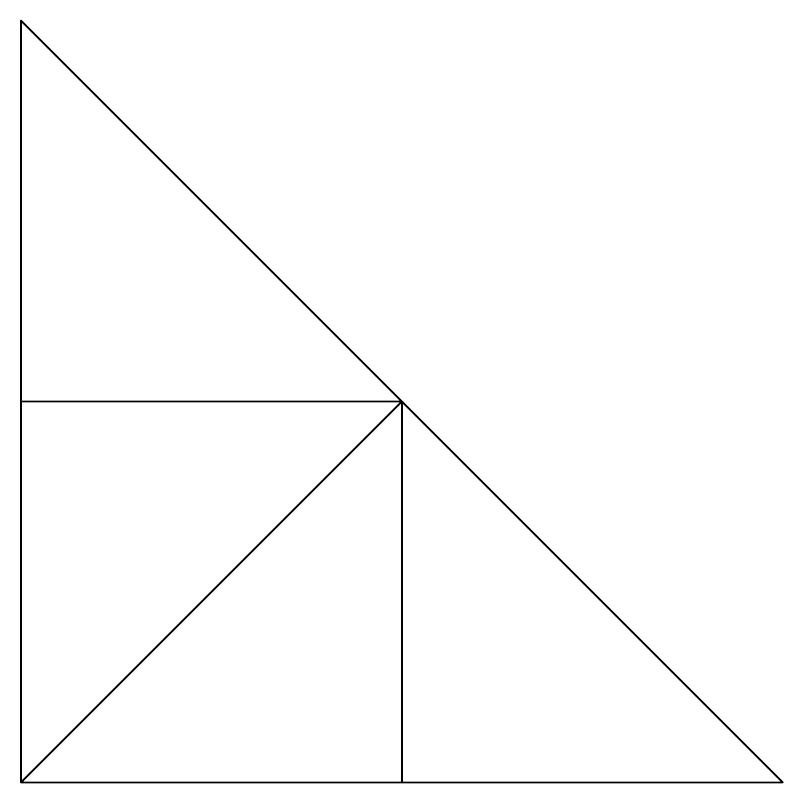}
\caption{Foldable figures corresponding to the reducible root system $A_1\times A_1$ (left) and the irreducible root system $B_2$ (right).}
\end{center}
\end{figure}
\noindent
The next theorem summarizes the connections between affine Weyl groups and their fundamental domains, and foldable figures. For proofs, see \cite{bour,hw}.
\begin{theorem}
\begin{enumerate}
\item   Let $\mathcal{G}$ be a reflection group and $\mathcal{O}_n$ the group of linear isometries of $\mathbb{E}^n$. Then there exists a homomorphism $\phi: \mathcal{G}\to\mathcal{O}_n$ given by $\phi(g)(x) = g(x) - g(0)$, $g\in\mathcal{G},\; x\in\mathbb{E}^n$.
The group $\mathcal{G}$ is called essential if $\phi(\mathcal{G})$ only fixes $0\in\mathbb{E}^n$. The elements of $\ker\phi$ are called translations. If $\mathcal{G}$ is essential and without fixed points then $\mathcal{G}$ has a compact fundamental domain.
\item  The reflection group generated by the reflections about the bounding hyperplanes of a foldable figure $F$ is the affine Weyl group $\widetilde{\mathcal{W}}$ of some root system. Moreover, $\widetilde{\mathcal{W}}$ has $F$ as a fundamental domain.
\item There exists a one-to-one correspondence between foldable figures and reflection groups that are essential and without fixed points.
\end{enumerate}
\end{theorem}

Note that in view of this theorem, our construction of affine fractal hypersurfaces took place over a foldable figure, namely $\triangle$ together with its collection of subfigures $\{\triangle_i\st i\in \N_N\}$. As the foldable figure and hence every subfigure of it is the fundamental domain of an associated affine Weyl group, the existence of a labelling map $\ell$ is guaranteed. In addition -- and this will become important in the next section -- we also constructed an orthonormal basis on this foldable figure consisting of a finite number of affine basis fractal hypersurfaces. 

In the next section it will become necessary to consider dilates of the fundamental domain of an affine Weyl group, i.e., a foldable figure. To this end, let $F\subset\mathbb{R}^n$ be a foldable figure with $0\in\mathbb{R}^n$ as one of its vertices. Denote by $\mathcal{H}$ be the set of
hyperplanes associated with $F$ and by $\Sigma$ be the tessellation of $F$ induced by $\mathcal{H}$. The affine Weyl group of the foldable figure
$F$, $\widetilde{\mathcal{W}}$, is then the group generated by the affine reflections $r_H$, where $H\in \mathcal{H}$. 

Now, fix $1 < \varkappa\in\mathbb{N}$ and define $\triangle := \varkappa F$. Then $\triangle$ is also a foldable figure, whose $N := \varkappa^n$ subfigures $\triangle_i\in\Sigma$. Assume w.l.o.g that $\triangle_1 = F$. The tessellation and set of hyperplanes for $\triangle$ are $\varkappa \Sigma$ and $\varkappa\mathcal{H}$, respectively. Moreover, the affine reflection group generated by $\varkappa\mathcal{H}$ is an isomorphic subgroup of $\widetilde{\mathcal{W}}$. Note that the similarity ratio $a = 1/\varkappa$. By simple transitivity of $\widetilde{\mathcal{W}}$, define
similitudes $u_i:\triangle\to\triangle_i$ by:
\[
u_1 := (1/\varkappa)(\cdot)\quad\text{and}\quad\forall j = 2,\ldots, N:\,\quad u_j := r_{j,1}\circ u_1.
\]
Now the construction proceeds as in Section \ref{sec4}. This yields a fractal function $\ff$ defined over $\triangle$ and an orthonormal basis of the form \eqref{444} whose elements are defined over the foldable figure $\varkappa F$.

Given a fractal function $\ff$ defined over a foldable figure $F$, we can extend $\ff$ to all of $\mathbb{R}^n$ by using the fact that the foldable figure $F$ is a fundamental domain for its associated affine Weyl group and that it tessellates $\mathbb{R}^n$ by reflections in its bounding hyperplanes, i.e., under the action of $\widetilde{\mathcal{W}}$. For any foldable figure $F'\in \Sigma$ there exists a unique $r\in \wW$ so that $F' = r(F)$. Define an extension $\overline{\ff}:\R^n\to\R$ of $\ff$ by
\[
\overline{\ff}\vert_{F'} := \ff\circ r^{-1}.
\]
Similarly, we define an orthonormal affine fractal basis $\fB'$ over $F'$ by setting $\fB' := \{\fb\circ r^{-1}\st \fb\in \fB\}$. Hence, we obtain 
$$ 
\bigcup_{r\in \wW} \{\fb\circ r^{-1}\st \fb\in \fB\}
$$
as an orthonormal affine fractal basis for $L^2(\R^n)$.

%\[
%\A_n^{\widetilde{\mathcal{W}}}:= \prod_{r\in \wW}\A_n.
%\]
%For $\boldsymbol{\Lambda}\in \A_n^{\widetilde{\mathcal{W}}}$, define $f_{\boldsymbol{\Lambda}}:\R^n\to\R$ by
%\[
%f_{\boldsymbol{\Lambda}}\vert_{r( \overset{\circ}{\triangle})} := f_{\boldsymbol{\Lambda}(r)}\circ r^{-1},\quad r\in\widetilde{\mathcal{W}},
%\]
%where $\boldsymbol{\Lambda}(r) = (\boldsymbol{\Lambda}(r)_1,\ldots,\boldsymbol{\Lambda}(r)_N)$ is the $r$-th coordinate of $\boldsymbol{\Lambda}$. The values of $f_{\boldsymbol{\Lambda}}$ are left unspecified on the hyperplanes $\varkappa\mathcal{H}$, a set of Lebesgue measure zero in $\mathbb{R}^n$, and thus $f_{\boldsymbol{\Lambda}}$ actually represents an equivalence class of functions.

%
\section{Wavelet Sets}\label{sec6}
In this section, we establish a connection between the various concepts introduced in the previous sections and the multiscale structure of wavelets. Our emphasis will be entirely on wavelet sets in $\R^n$ and most of the material found here is taken from \cite{dai1,dai2,dai3,lm}. The reader interested in a more operator-theoretic formulation of wavelet sets is referred to \cite{dai2}. For motivational purposes, we present first the one-dimensional setting, i.e., wavelet sets on $\R$, and then proceed to the higher-dimensional scenario.

To this end, we need to first define wavelets as follows.
\begin{definition}
A \textbf{dyadic orthonormal wavelet} on $\R$ is a unit vector $\psi \in L^2(\mathbb{R}, m)$, where $m$ denotes Lebesgue measure, with the property that the set
\begin{equation}\label{psi}
\{2^{\frac{n}{2}}\psi(2^n t - \ell) \st n,\ell \in \mathbb{Z}\}
\end{equation}
of all integral translates of $\psi$ followed by dilations by arbitrary integral powers of $2$, is an orthonormal basis for $L^2(\mathbb{R},m)$.
\end{definition}
\noindent
We remark that this is not the most general definition of wavelet but for our purposes it is sufficient.

For later developments, we require the following two operators. Let $T$ denote the unitary translation and $D$ the unitary dilation operator in $B(L^2(\mathbb{R}))$, the Banach space of bounded linear operators from $L^2(\mathbb{R})$ to itself, defined by 
$$
(Tf)(t) := f(t-1)\quad\text{and}\quad (Df)(t) := \sqrt{2}f(2t), \quad f\in L^2(\R), \;t\in \R.
$$
Then, we may write \eqref{psi} more succinctly as
$$
2^{\frac{n}{2}}\psi(2^nt - \ell) = (D^nT^\ell \psi)(t)
$$
for all $n,\ell \in \mathbb{Z}$, $t\in \R$. Note that $TD = DT^2$.
 
Next, we introduce the Fourier transform on $\R$. For $f,g\in L^1(\mathbb{R}) \cap L^2(\mathbb{R})$, let
\begin{equation}\label{eq18}
(\mathscr{F}f)(s) := \frac1{\sqrt{2\pi}} \int_{\mathbb{R}} e^{-ist}
f(t)dt =: \widehat f(s),
\end{equation}
and
\begin{equation}\label{eq19}
(\mathscr{F}^{-1}g)(t) = \frac1{\sqrt{2\pi}} \int_{\mathbb{R}} e^{ist}g(s)ds = \widehat{g}(-s) =: g^\vee (t).
\end{equation}
\nl
Note that this particular form of the Fourier-Plancherel transform defines a unitary operator on $L^2(\R)$. In particular, we remark that
\[
(\mathscr{F} T)(f)(s) := \frac{1}{\sqrt{2\pi}} \int_\R e^{-ist} f(t-1)dt = e^{-is} (\mathscr{F}f)(s) =: e^{-is} \widehat{f}(s).
\]
Hence, $\sF T \sF^{-1} f = e^{-is} f =: M_{e^{-is}} f$. If we define $\widehat{T} := \sF T \sF^{-1}$ then $\widehat{T} = M_{e^{-is}}$.
Similarly, we obtain for any $n\in \Z$
\begin{align*}
(\mathscr{F}D^nf)(s) &= \frac1{\sqrt{2\pi}} \int_{\mathbb{R}} e^{-ist} (\sqrt{2})^n f(2^nt)dt = (\sqrt 2)^{-n}\,\frac1{\sqrt{2\pi}} \int_{\mathbb{R}}e^{-i2^{-n}st} f(t)dt\\
&= (\sqrt 2)^{-n} (\mathscr{F}f)({2^{-n}s}) =
(D^{-n}\mathscr{F}f)(s).
\end{align*}
As above, if we set $\widehat{D^n} := \sF D^n \sF^{-1}$ then $\widehat{D^n} = D^{-n}$ and therefore $\widehat{D} = D^{-1}$.

Wavelet sets belong to the theory of wavelets via the Fourier transform. 
\begin{definition}
A \textbf{wavelet set} in $\mathbb{R}$ is a measurable subset $E$ of $\mathbb{R}$ for which
$\frac1{\sqrt{2\pi}} \chi_E$ is the Fourier transform of a wavelet. 
\end{definition}
\noindent
The wavelet $\widehat\psi_E := \frac1{\sqrt{2\pi}}\chi_E$ is sometimes called \textbf{$s$--elementary}. The class of wavelet sets was also investigated in \cite{HWW,FW}. In their theory the corresponding wavelets are called MSF (minimally supported frequency) wavelets.

\begin{example}
The prototype of a wavelet set is the Shannon set given by
\begin{equation}\label{eq33}
E_S = [-2\pi, -\pi) \cup [\pi,2\pi).
\end{equation}
The orthonormal wavelet for $E_S$ is then given by $\psi_S (t) = 2 \sinc (2t -1) - \sinc t$.
\begin{figure}[h!]
\begin{center}
\includegraphics[width = 6cm, height = 3cm]{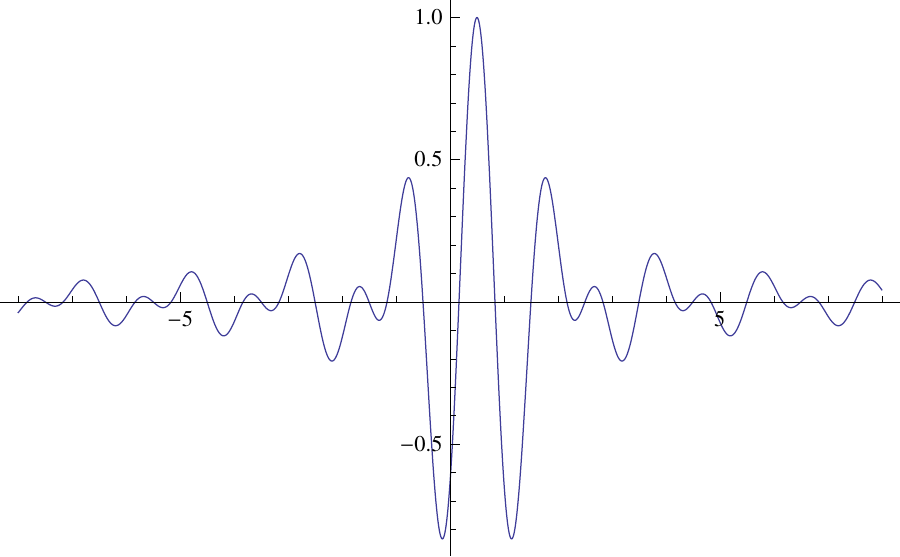}
\end{center}
\caption{The Shannon wavelet.}
\end{figure}

To prove that $\psi_S$ is indeed an orthonormal wavelet, note that the set of exponentials
\[
\{e^{i\ell s}\st \ell\in \mathbb{Z}\}
\]
restricted to $[0,2\pi]$ and normalized by $\frac1{\sqrt{2\pi}}$ is an orthonormal basis for $L^2[0,2\pi]$. Write $E_S = E_-\cup E_+$ where $E_- = [-2\pi, -\pi)$, $E_+ = [\pi,2\pi)$. Since $\{E_- +2\pi, E_+\}$ is a partition of $[0,2\pi)$ and since the exponentials $e^{i\ell s}$ are invariant under translation by $2\pi$, it follows that
\begin{equation*}
\left\{\frac{e^{i\ell s}}{\sqrt{2\pi}}\Big|_{E_S}\st \ell\in
\mathbb{Z}\right\}
\end{equation*}
is an orthonormal basis for $L^2(E_S)$. Since $\widehat T = M_{e^{-is}}$, this set can be written as
\begin{equation*}
\{\widehat T^\ell \widehat\psi_S\st \ \ell\in \mathbb{Z}\}.
\end{equation*}
Next, note that any ``dyadic interval'' of the form $J = [b,2b)$, for some $b>0$ has the property that $\{2^nJ\st n\in\mathbb{Z}\}$, is a partition of $(0,\infty)$. Similarly, any set of the form
\begin{equation*}
\mathcal{K} = [-2a,-a)\cup [b,2b)
\end{equation*}
for $a,b>0$, has the property that
\[
\{2^n\mathcal{K}\st \ n\in \mathbb{Z}\}
\]
is a partition of $\mathbb{R}\backslash\{0\}$. 

To complete the proof, we need to introduce one more item.
\begin{definition}
Let $U$ be a given unitary operator on a Hilbert space $\mathcal{H}$. A nonempty subspace $\mathcal{K}$ of $\mathcal{H}$ is called a \textbf{wandering subspace} for $U$ if $U^m (\mathcal{K}) \bot U^n (\mathcal{K})$, for all $m\neq n\in \N$. If, in addition, $\mathcal{H} = \bigoplus_{n\in\Z} U^n (\mathcal{K})$, then we say that $\mathcal{K}$ is a \textbf{complete wandering subspace} of $\mathcal{H}$ for $U$. 
\end{definition}
It follows that the space $L^2(\mathcal{K})$, considered as a subspace of $L^2(\mathbb{R})$, is a complete wandering subspace for the dilation unitary $(Df)(s) = \sqrt 2\ f(2s)$. For each $n\in \mathbb{Z}$,
\begin{equation*}
D^n(L^2(\mathcal{K})) = L^2(2^{-n}\mathcal{K}).
\end{equation*}
So $\displaystyle{\bigoplus_{n\in\Z}} D^n(L^2(\mathcal{K}))$ is a direct sum decomposition of $L^2(\mathbb{R})$. In particular $E_S$ has this property. Thus,
\begin{equation*}
D^n\left\{\frac{e^{i\ell s}}{\sqrt{2\pi}}\Big|_{E_S}\st \ \ell\in\mathbb{Z}\right\} = \left\{\frac{e^{2^ni\ell s}}{\sqrt{2\pi}}\Big|_{2^{-n}E_S} \st \ \ell\in
\mathbb{Z}\right\}
\end{equation*}
basis for $L^2(2^{-n}E_S)$ for each $n$. Hence
\[
\{D^n\widehat T^\ell \widehat\psi_S\st \ n,\ell\in \mathbb{Z}\} = \{\widehat{D^n}\widehat T^\ell \widehat\psi_S\st \ n,\ell\in \mathbb{Z}\}
\]
is an orthonormal basis for $L^2(\mathbb{R})$, as required.
\end{example}

It follows from the above arguments that {\em sufficient\/} conditions for $E$ to be a wavelet set are:
\begin{itemize}
\item[(W1)] The set of normalized exponentials $\left\{\frac1{\sqrt{2\pi}} e^{i\ell s} \st \ell\in \mathbb{Z}\right\}$, when restricted to $E$, constitutes an orthonormal basis for $L^2(E)$.
\item[(W2)] The family $\{2^nE\st n\in\mathbb{Z}\}$ of dilates of $E$ by integral
powers of 2 should constitute a measurable partition (i.e., a partition modulo
null sets) of $\mathbb{R}$.
\end{itemize}

These observations now motivate the next definitions.
\begin{definition}
Two measurable sets $E,F \subseteq\R$ are called \textbf{translation congruent modulo\/}
$2\pi$
\begin{itemize}
\item if there exists a measurable bijection $\phi : E\to F$ such that
$\phi(s)-s = n(s)\, 2\pi$, for each $s\in E$, and a unique $n(s)\in\Z$,
\end{itemize}
or, equivalently,
\begin{itemize}
\item if there is a measurable partition $\{E_n\st n\in
\mathbb{Z}\}$ of $E$ such that
\begin{equation*}\label{eq39}
\{E_n  + n\,2\pi\st \ n\in \mathbb{Z}\}
\end{equation*}
is a measurable partition of $F$. 
\end{itemize}
Two measurable sets $G, H\subseteq \R$ are called \textbf{dilation congruent modulo\/} 2
\begin{itemize} 
\item if there is a measurable bijection $\tau : G\to H$ such that for each $s\in G$ there exists an $n = n(s)\in\Z$ with $\tau(s) = 2^ns$,
\end{itemize}
or equivalently,
\begin{itemize}
\item if there is a
measurable partition $\{G_n\st n\in \Z\}$ of $G$ such that
\begin{equation*}\label{eq40}
\{2^nG_n\st n\in \Z\}
\end{equation*}
is a measurable partition of $H$.
\end{itemize}
\end{definition}

\begin{example}: 
$E := [0, 2\pi]$ and $F:= [-2\pi, -\pi] \cup [\pi, 2\pi]$.
\end{example}

The following lemma is proved in \cite{dai1}.
\begin{lemma}
Let $f\in L^2(\mathbb{R})$, and let $E = \text{\rm supp}(f)$. Then $f$ has the
property that
\[
\{e^{i\ell s}f\st \ell\in \mathbb{Z}\}
\]
is an orthonormal basis for $L^2(E)$ if and only if
\begin{itemize}
\item[(i)] $E$ is congruent to $[0,2\pi)$ modulo $2\pi$, and
\item[(ii)] $|f(s)| = \frac1{\sqrt{2\pi}}$ a.e.\ on $E$.
\end{itemize}
\end{lemma}
Observe that if $E$ is $2\pi$--translation congruent to $[0,2\pi)$, then since
\[
\{[0,2\pi) + 2\pi n\st n\in \mathbb{Z}\}
\]
is a measurable partition of $\mathbb{R}$, so is
\[
\{E + 2\pi n\st n\in \mathbb{Z}\}.
\]
Similarly, if $F$ is  $2$--dilation congruent to the Shannon wavelet set $E_S = [-2\pi, -\pi) \cup [\pi, 2\pi)$, then since $\{2^nE_S : n \in \mathbb{Z}\}$ is a measurable partition of $\mathbb{R}$, so is $\{2^n F : n \in \mathbb{Z}\}$.

We say that a measurable subset $G\subseteq \mathbb{R}$ is a 2--\textbf{dilation
generator\/} of a {\em partition\/} of $\mathbb{R}$ if the sets
\begin{equation*}\label{eq41}
2^nG := \{2^ns\st s\in G\},\qquad n\in \mathbb{Z}
\end{equation*}
are disjoint and $\mathbb{R}\setminus \displaystyle{\bigcup_{n\in \Z}} 2^nG$ is a null set. 
\nl
Analogously, we say
that $E\subseteq \mathbb{R}$ is a $2\pi$--\textbf{translation generator} of a {\em 
partition\/} of $\mathbb{R}$ if the sets
\begin{equation*}\label{eq42}
E + 2n\pi := \{s + 2 n\pi\st s\in E\},\qquad n\in \mathbb{Z},
\end{equation*}
are disjoint and $\mathbb{R}\setminus \displaystyle{\bigcup_{n\in \Z}} (E+2n\pi)$ is a null set.

The next theorem gives necessary and sufficient conditions for a measurable set $E\subseteq\R$ to be a wavelet set. Before we state it, we need a definition.

\begin{definition}
A \textbf{fundamental domain} for a group of (measurable) transformations $\mathcal{G}$ on a measure space $(\Sigma,\mu)$ is a measurable set $C$ with the property that $\{g(C) : g\in\mathcal{G}\}$ is a measurable partition (tessellation) of $\Sigma$; that is, $\Sigma \setminus \left(\displaystyle{\bigcup_{g\in\mathcal{G}}} g(C)\right)$ is a $\mu$-null set and $g_1 (C) \cap g_2 (C)$ is a $\mu$-null set for $g_1\neq g_2$.
\end{definition}

\begin{theorem}
Let $E\subseteq\mathbb{R}$ be a measurable set. Then $E$ is a wavelet set if and only if one of the following equivalent conditions holds. 
\begin{itemize}
\item[(i)] $E$ is both a $2$--dilation generator of a partition (modulo null sets) of $\mathbb{R}$ and a $2\pi$--translation generator of a partition (modulo null sets) of $\mathbb{R}$.
\item[(ii)] $E$ is both translation congruent to $[0,2\pi)$ mod $2\pi$ and dilation congruent to $[-2\pi,-\pi) \cup [\pi,2\pi)$ mod $2$.
\item[(iii)] $E$ is a fundamental domain for the dilation group $\langle D^n \st n\in \Z\rangle$ and at the same time a fundamental domain for the translation group $\langle T^k_{2\pi} \st k\in \Z\rangle$. Here $T_{2\pi}$ is translation by $2\pi$ along the real axis.)
\end{itemize}
\end{theorem}
\begin{proof}
See \cite{dai1}.
\end{proof}

Now we like to extend the above concepts and definitions to $\R^n$. We will do this in a slightly more general setting. To this end, let $\X$ be a metric space and $\mu$ a $\sigma$-finite non-atomic Borel measure on $\X$ for which the measure of every open set is positive and for which bounded sets have finite measure. Let $\mathcal{T}$ and $\mathcal{D}$ be countable groups of homeomorphisms of $\X$ that map bounded sets to bounded sets and which are absolutely continuously in the sense that they map $\mu$-null sets to $\mu$-null sets. Furthermore, let $\mathcal{G}$ be a countable group of absolutely continuous Borel measurable bijections of $\X$. Denote by $\mathcal{B}$ the family of Borel sets of $\X$.

The following definition completely generalizes the definitions of $2\pi$--translation congruence and $2$--dilation congruence given above.
\begin{definition}
Let $E, F\in\mathcal{B}$. We call $E$ and $F$ {\em $\mathcal{G}$--congruent} and write $E \sim_{\mathcal{G}} F$, if there exist measurable partitions
$\{E_g : g\in\mathcal{G}\}$ and $\{F_g : g\in\mathcal{G}\}$ of $E$ and $F$, respectively, such that $F_g = g (E_g)$, for all $g\in\mathcal{G}$,
modulo $\mu$-null sets.
\end{definition}
This definition immediately entails the next two results.
\begin{proposition}\label{prop99}
\begin{enumerate}
\item $\mathcal{G}$--congruence is an equivalence relation on the family of $m$-measurable sets.
\item If $E$ is a fundamental domains for $\mathcal{G}$, then $F$ is a fundamental domain for $\mathcal{G}$ iff $F \sim_{\mathcal{G}} E$.
\end{enumerate}
\end{proposition}

\begin{proof}
See \cite{dai1}.
\end{proof}

\begin{definition}\label{dil-trans}
We call $(\mathcal{D},\mathcal{T})$ an \textbf{abstract dilation--translation pair} if
\begin{enumerate}
\item   for each bounded set $E$ and each open set $F$ there exist elements $\delta\in\mathcal{D}$ and $\tau\in\mathcal{T}$ such that $\tau(F)
        \subset\delta(E)$;
\item   there exists a fixed point $\theta\in \X$ for $\mathcal{D}$ with the property that if $N$ is any neighborhood of $\theta$ and $E$ any bounded
        set, there is an element $\delta\in\mathcal{D}$ such that $\delta(E)\subset N$.
\end{enumerate}
\end{definition}

The following result and its proof can be found in \cite{dai2}.

\begin{theorem}\label{t1}
Let $\X$, $\mathcal{B}$, $\mu$, $\mathcal{D}$, and $\mathcal{T}$ as above. Let $(\mathcal{D},\mathcal{T})$ be an abstract dilation--translation pair with $\theta$ being the fixed point of $\mathcal{D}$. Assume that $E$ and $F$ are bounded measurable sets in $\X$ such that $E$ contains a neighborhood of $\theta$, and $F$ has non-empty interior and is bounded away from $\theta$. Then there exists a measurable set $G\subset X$, contained in
$\displaystyle{\bigcup_{\delta\in\mathcal{D}}} \delta(F)$, which is both $\mathcal{D}$--congruent to $F$ and $\mathcal{T}$--congruent to $E$.
\end{theorem}

The following is a consequence of Proposition \ref{prop99} and Theorem \ref{t1} and is the key to obtaining wavelet sets.
\begin{corollary}
With the terminology of Theorem \ref{t1}, if in addition $F$ is a fundamental domain for $\mathcal{D}$ and $E$ is a fundamental domain for
$\mathcal{T}$, then there exists a set $G$ which is a common fundamental domain for both $\mathcal{D}$ and $\mathcal{T}$.
\end{corollary}
In order to apply the above result to wavelet sets in $\mathbb{R}^n$, we require the following two definitions.

\begin{definition}
Let $A\in M_n(\R)$ be an $(n\times n)$--matrix with real coefficients. By an \textbf{orthonormal $(D_A, T)$--wavelet} we mean a function $\psi \in L^2(\mathbb{R}^n)$ such that
\begin{equation}\label{psin}
\{|\det(A)|^{\frac{n}{2}} \psi(A^n t - \ell) \st n\in \mathbb{Z},\, \ell\in\mathbb{Z}^n,\, i = 1,\ldots, n\}
\end{equation}
where $\ell = (\ell_1, \ell_2, ..., \ell_n)^\top$, is an orthonormal basis for $L^2(\mathbb{R}^n ; m)$.  (Here $m$ is product Lebesgue measure, and the superscript ${}^\top$ means transpose.)
\end{definition}

If $A \in M_n(\mathbb{R})$ is invertible (so in particular if $A$ is expansive), then the operator defined by
\begin{equation*}
(D_Af)(t) = |\det A|^{\frac12} f(At)
\end{equation*}
for $f \in L^2(\mathbb{R}^n)$, $t \in \mathbb{R}^n$, is \emph{unitary}. For $1 \leq i \leq n$, let $T_i$ be the unitary
operator determined by translation by $1$ in the $i^{th}$ coordinate direction.  Then the set \eqref{psin} above can be written as
\begin{equation*}
\{D^k_A T^{\ell} \psi \st k,\ell \in \mathbb{Z}^n\},
\end{equation*}
with $T^\ell := T^{\ell_1}_1 \cdot\cdot\cdot T^{\ell_n}_n$.

\begin{definition}
A \textbf{$(D_A, T)$--wavelet set} is a measurable subset $E$ of $\mathbb{R}^n$ for which the inverse Fourier transform of $(m(E))^{-n/2}\,\chi_E$ is an orthonormal $(D_A, T)$--wavelet.
\end{definition}

Two measurable subsets $H$ and $K$ of $\mathbb{R}^n$ are called \textbf{$A$--dilation congruent}, in symbols $H\sim_{\delta_A} K$, if there exist measurable partitions $\{H_\ell \st \ell\in\mathbb{Z}\}$ of $H$ and $\{K_\ell \st \ell\in\mathbb{Z}\}$ of $K$ such that $K_\ell = A^\ell H_\ell$ modulo Lebesgue null sets. Moreover, two measurable sets $E$ and $F$ of $\mathbb{R}^n$ are called \textbf{$2\pi$--translation congruent}, written $E\sim_{\tau_{2\pi}} F$, if there exist measurable partitions $\{E_\ell \st \ell\in\mathbb{Z}^n\}$ of $E$ and $\{F_\ell \st \ell\in\mathbb{Z}^n\}$ of $F$ such that $F_\ell = E_\ell + 2\pi\ell$ modulo Lebesgue null sets.

We remark that this generalizes to $\R^n$ the previous definition of $2\pi$--translation congruence for subsets of $\R$.  Observe that $A$--dilation by an expansive matrix together with $2\pi$--translation congruence is a special case of an abstract dilation-translation pair as introduced in Definition \ref{dil-trans}.  Let $\mathcal{D} := \langle A^k \st k \in \Z\rangle$ be the dilation group generated by powers of $A$, and let $\mathcal{T} := \langle T_{2\pi }^\ell \st \ell \in \Z^n \rangle$ be the group of translations generated by the translations $T_{2\pi }$ along the coordinate directions.  Let $E$ be any bounded set and let $F$ be any open set that is bounded away from 0. Let $r > 0$ be such that  $E \subseteq B_r(0)$.  Since $A$ is expansive there is an $\ell \in \mathbb{N}$ such that $A^\ell F$ contains a ball $B$ of radius large enough so that $B$ contains some lattice point  $2k\pi$ together with the ball $B_R(2k\pi)$ of radius $R > 0$ centered at the lattice point. Then $E + 2k\pi \subseteq A^\ell F$.  That is, the $2k\pi$--translate of $E$ is contained in the $A^\ell$--dilate of $F$,  as required in (1) of Definition \ref{dil-trans}. For (2) of Definition \ref{dil-trans}, let $\theta = 0$, and let $N$ be a neighborhood of 0, and let $E$ be any bounded set.  As above, choose $r  > 0$ with  $E \subseteq B_r(0)$. Let $\ell \in \mathbb{N}$ be such that $A^\ell N$ contains $B_r(0)$.  Then $A^{-\ell}$ is the required dilation such that $A^{-\ell}E \subseteq N$.
\par
Note that if $W$ is a measurable subset of $\mathbb{R}^n$ that is $2\pi$--translation congruent to the $n$-cube $E := \underset{i=1}{\overset{N}{\times}}  [-\pi, \pi)$, it follows from the exponential form of $\widehat{T}_j$ that $\left\{\widehat{T}_1^{\ell_1}\widehat{T}_2^{\ell_2}\cdots
\widehat{T}_n^{\ell_n}\,(m(W))^{-1/2}\,\chi_W\st \ell = (\ell_1,\ell_2,\ldots,\ell_n)\in\mathbb{Z}^n\right\}$ is an orthonormal basis for $L^2 (W)$.
Furthermore, if $A$ is an expansive matrix and $B$ the unit ball of $\mathbb{R}^n$ then with $F_A := A(B)
\setminus B$ the collection $\{A^k F_A : k\in\mathbb{Z}\}$ is a partition of $\mathbb{R}^n\setminus\{0\}$. Consequently, $L^2 (F_A)$, considered as a subspace of $L^2 (\mathbb{R}^n)$, is a complete wandering subspace for $D_A$. Hence, $L^2 (\mathbb{R}^n)$ is a direct sum decomposition of the subspaces $\{D_A^k L^2 (F_A) \st k\in\mathbb{Z}\}$. Clearly, any other measurable set $F^\prime\sim_{\delta_A} F_A$ has this same property.

The following theorem gives the existence of wavelet sets in $\mathbb{R}^n$. For the proof, see \cite{dai3}.
\begin{theorem}
Let $n\in\mathbb{N}$ and let $A$ be an expansive $n\times n$ matrix. Then there exist $(D_A, T)$--wavelet sets.
\end{theorem}

\begin{example}
An example of a (fractal) wavelet set in $\R^2$ is shown below. The exact construction parameters are given in \cite{dai1}.
\begin{figure}[h!]
\begin{center}
\includegraphics[width = 4cm, height = 4cm]{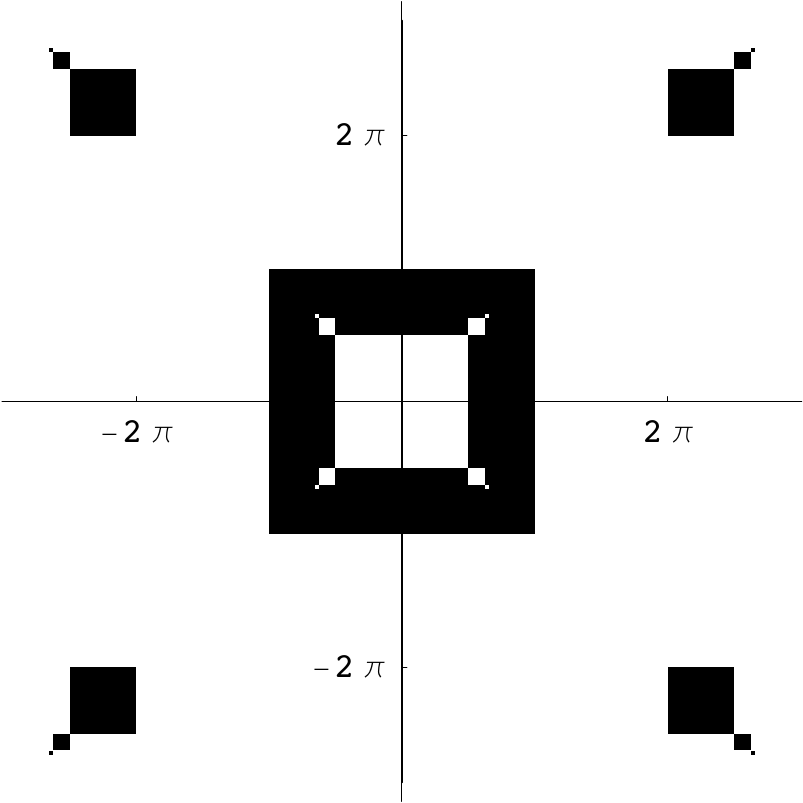}
\end{center}
\end{figure}
\end{example}
For more examples and constructions of wavelet sets, we encourage the reader to consult, for instance, \cite{bmm,bl,dai3,merrill1,merrill2,sw}.
\vskip 6pt
Finally, we put the concepts of fractal hypersurface, foldable figure, affine Weyl group and wavelet together and introduce a new type of wavelet set is called a \textbf{dilation-reflection wavelet set}. The idea is to adapt Definition \ref{dil-trans}, replacing the group of translations $\mathcal{T}$ in the traditional wavelet theory by an affine Weyl group $\widetilde{\cW}$ whose fundamental domain is a foldable figure $C$, and to use the orthonormal basis of affine fractal hypersurfaces constructed in Section \ref{sec4}.

In Definition \ref{dil-trans}, we take $X := \mathbb{R}^n$ endowed with the Euclidean affine structure and distance, and for the abstract translation group $\mathcal{T}$ we take the affine Weyl group $\widetilde{{\mathcal{W}}}$ generated by a group of affine reflections arising from a locally finite collection of affine hyperplanes of $X$. Let $C$ denote a fundamental domain for $\widetilde{{\mathcal{W}}}$ which is also a foldable figure. Recall that $C$ is a simplex, i.e., a convex connected polytope, which tessellates $\mathbb{R}^n$ by reflections about its bounding hyperplanes. Let $\theta$ be any fixed interior point of $C$. Let $A$ be any real expansive matrix in $M_n(\mathbb{R})$ acting as a linear transformation on $\mathbb{R}^n$.  In the case where $\theta$ is the orgin $0$ in $\mathbb{R}^n$ we simply take $D_A$ to be the usual dilation by $A$ and the abstract dilation group to be $\mathcal{D} = \{D^k_A \st k \in \mathbb{Z}\}$. For a general $\theta$, define $D_{A,\theta}$ to be the affine mapping $D_\theta(x) := A(x - \theta) + \theta, x \in\mathbb{R}^n$ and $\mathcal{D}_{\theta} = \{D^k_{A,\theta} \st k \in \mathbb{Z}\}$.

\begin{proposition}\label{th10}
 $(\mathcal{D}_\theta,\widetilde{{\mathcal{W}}})$ is an
abstract dilation-translation pair in the sense of Definition \ref{dil-trans}.
\end{proposition}
\begin{proof}
The proof is given in \cite{lm} but for completeness, we repeat it here. By the definition of $D$, $\theta$ is a fixed point for $\mathcal{D}_\theta$. By a change of coordinates we may assume without loss of generality that $\theta = 0$ and consequently that $D$ is multiplication by $A$ on $\mathbb{R}^n$.
\par
Let $B_r (0)$ be an open ball centered at $0$ with radius $r > 0$ containing both $E$ and $C$. Since $F$ is open and $A$ is expansive, there exists a $k \in\N$ sufficiently large so that $D^k F$ contains an open ball $B_{3r}(p)$ of radius $3r$ and with some center $p$.  Since $C$ tiles $\R^n$ under the action of $\wW$, there exists a word $w\in \wW$ such that $w(C) \cap B_r(p)$ has positive measure. (Note here that $B_r(p)$ is the ball with the same center $p$ but with smaller radius $r$.) Then $w(B_r(0)) \cap B_r(p) \neq \emptyset$. Since reflections (and hence words in $\wW$) preserve diameters of sets in $\R^n$, it follows that $w(B_r(0))$ is contained in $B_{3r}(p)$.  Hence $w(E)$ is contained in $D^k(F)$, as required.
\par
This establishes part (1) of Definition \ref{dil-trans}. Part (2) follows from the fact that $\theta = 0$ and $D$ is multiplication by an expansive matrix in $M_n(\mathbb{R})$.
\end{proof}

Now we extend the definition of $(D_A,T)$--wavelet set in $\R^n$ to this new setting.

\begin{definition} \label{def7.2}
Given an affine Weyl group $\widetilde{{\mathcal{W}}}$ acting on $\mathbb{R}^n$ with fundamental domain a foldable figure $C$, given a designated interior point $\theta$ of $C$, and given an expansive matrix $A$ on $\mathbb{R}^n$, a \textbf{$(D_{A,\theta}, \wW)$--wavelet set} is a measurable subset $E$ of $\mathbb{R}^n$ satisfying the properties:
\begin{enumerate}
\item $E$ is congruent to $C$ (in the sense of Definition 2.4) under the
action of $\widetilde{{\mathcal{W}}}$, and
\item $W$ generates a measurable partition of $\mathbb{R}^n$ under the action of
the affine mapping $D(x) := A(x - \theta) + \theta$.
\end{enumerate}
In the case  where $\theta = 0$, we abbreviate $(D_{A,\theta}, \wW)$ to $(D_A, \wW)$.
\end{definition}

The next result establishes the existence of $(D_{A,\theta}, \wW)$--wavelet sets. The proof is a direct application of Theorem \ref{t1} and can be found in \cite{lm}.
\begin{theorem}\label{100}
There exist $(D_{A,\theta}, \wW)$--wavelet sets for every choice of $\widetilde{W}$, A,  and $\theta$.
\end{theorem}

In the case of a dilation--translation wavelet set $W$, the two systems of unitaries are $\mathcal{D} := \{D_A^k \st k\in\mathbb{Z}\}$, where $A\in M_n (\mathbb{R})$ is an expansive matrix, and $\mathcal{T} := \{T^\ell \st \ell\in\mathbb{Z}^n\}$. An orthonormal wavelet basis of $L^2 (\mathbb{R}^n)$ is then obtained by setting $\widehat{\psi}_W := (m(W))^{-1/2} \chi_W$ and taking
\be\label{classical}
\left\{\widehat{D}_A^k \widehat{T}^\ell \widehat{\psi}_W \st k\in\mathbb{Z},\, \ell\in\mathbb{Z}^n\right\}.
\ee
\par
For the systems of unitaries  $\mathcal{D} := \{D_A^k \st k\in\mathbb{Z}\}$ and $\widetilde{\mathcal{W}}$, the affine Weyl group associated with a foldable figure $C$, one obtains as an orthonormal basis for $L^2 (\mathbb{R}^n)$
\be\label{DB}
\left\{D_{\varkappa I}^k \fB_r \st k\in\mathbb{Z}, \;r\in\wW\right\},
\ee
where  $\fB_r = \left\{\fb\circ r\st \fb\in \fB\right\}$ is an affine fractal hypersurface basis as constructed in the previous section.

It was shown in \cite{ghm1,ghm2} that one can even construct multiresolution analyses based on affinely generated fractal functions as constructed in Section 3 and 4. The resulting sets of scaling vectors and multiwavelets are piecewise affine fractal functions and they generate orthonormal bases for the underlying approximation, respectively, wavelet spaces. If one uses these orthonormal multiwavelet bases in \eqref{DB} to obtain orthonormal bases for $L^2(\R^n)$ then the analogy to the classical case as exemplified by \eqref{classical} is complete. (See \cite{lm}.)
\end{document}